\documentclass[preprint,11pt,numbers]{elsarticle-arxiv}

\usepackage{amssymb}
\usepackage{amsmath}
\usepackage{amsthm}
\usepackage{tikz}
\usetikzlibrary{calc}
\usepackage{url}
\usepackage{subcaption}

\setcitestyle{square,comma}

\newtheorem{theorem}{Theorem}[section]

\newtheorem{corollary}{Corollary}[theorem]
\newtheorem{conjecture}{Conjecture}[section]

\newtheorem{definition}{Definition}[section]
\newtheorem{lemma}{Lemma}[section]
\newtheorem{lemmacorollary}{Corollary}[lemma]

\newtheorem{observation}{Observation}[section]

\journal{Discrete Applied Mathematics}

\begin{document}
\bibliographystyle{alpha}
\begin{frontmatter}

\title{\textbf{\Large Proper Conflict-Free Choosability for Graphs with Bounded Average Degree}}

\author[1]{Zhijun Lu}
\author[1]{Qirui Ying}
\author[1]{Huimin Song\corref{cor1}}
\cortext[cor1]{Corresponding author. \textit{E-mail address:} \url{hmsong@sdu.edu.cn} (H.Song).}

\affiliation[1]{organization={School of Mathematics and Statistics},
            addressline={Shandong University},
            city={Weihai},
            postcode={264209},
            state={Shandong},
            country={China}}

\begin{abstract}
\ \ \ For a graph \(G\), a proper coloring of \(G\) is called proper conflict-free if for every non-isolated vertex \(u\), there is at least one color appearing exactly once in \(N_G(u)\). A graph \(G\) is proper conflict-free \(f\)-choosable if for every list assignment \(L\) with \(|L(v)|\ge f(v)\) for each vertex \(v\), \(G\) admits a proper conflict-free \(L\)-coloring.

Recently, Kashima, \v{S}krekovski, and Xu proposed a conjecture on proper conflict-free list coloring. For a graph \(G\), let \(\kappa_G:V(G)\to \mathbb{N}\) be defined by
\[
\kappa_G(v)=
\begin{cases}
4, & \text{if } d_G(v)=2,\\[4pt]
d_G(v)+1, & \text{if } d_G(v)\neq 2.
\end{cases}
\]
They conjectured that every connected graph other than \(C_5\) is proper conflict-free \(\kappa_G\)-choosable.

In this paper, we confirm this conjecture in two classes of graphs with bounded average degree, thereby generalizing results of Kashima, \v{S}krekovski, and Xu and of Wang and Zhang. We prove that every connected graph \(G\neq C_5\) with either \(\operatorname{mad}(G)<\frac{12}{5}\) or \(\Delta(G)\le3\) is proper conflict-free \(\kappa_G\)-choosable. To prove these results, we introduce a method based on systems of proper conflict-free representatives and develop a construction of auxiliary graphs that preserves the maximum average degree bound.
\end{abstract}


\end{frontmatter}

\section{Introduction}\label{s1}

\subsection{Preliminaries}

All graphs considered are simple, finite, and undirected. For a graph $G$, let $V(G)$ and $E(G)$ denote its vertex set (always nonempty) and edge set, respectively. The \emph{neighborhood} of a vertex $v$ in $G$ is $N_G(v)=\{u\in V(G):uv\in E(G)\}$, and its elements are called the \emph{neighbors} of $v$. The \emph{degree} of $v$, denoted by $d_G(v)$, is $|N_G(v)|$. A graph $H$ is a \emph{subgraph} of a graph $G$, denoted by $H\subseteq G$, if $V(H)\subseteq V(G)$ and $E(H)\subseteq E(G)$. 
A subgraph $H$ of $G$ is called an \emph{induced subgraph} of $G$ if $E(H)=\{uv\in E(G):u,v\in V(H)\}$. For a set $S\subseteq V(G)$, the subgraph of $G$ induced by $S$ is denoted by $G[S]$.
The \emph{degeneracy} of a graph $G$ is defined by $\max_{\emptyset\neq H\subseteq G}\delta(H)$, where the maximum is taken over all nonempty subgraphs $H$ of $G$. For a nonnegative integer $d$, a graph $G$ is called \emph{$d$-degenerate} if its degeneracy is at most $d$. The \emph{maximum average degree} of $G$, denoted $\operatorname{mad}(G)$, is defined as $\operatorname{mad}(G)=\max_{H\subseteq G,\,V(H)\neq\emptyset}\frac{2|E(H)|}{|V(H)|}$, where $H$ is a subgraph of $G$. The \emph{maximum degree} $\Delta(G)$ and the \emph{minimum degree} $\delta(G)$ are the maximum and minimum degree among all vertices of $G$, respectively. The \emph{girth} $g(G)$ is the length of a shortest cycle in $G$, and $g(G)=\infty$ if $G$ is acyclic. We denote by $C_5$ the 5-cycle. When no ambiguity arises, we abbreviate $d_G(v)$, $N_G(v)$, and $g(G)$ as $d(v)$, $N(v)$, and $g$, respectively.

For $k\in\mathbb{N}$, a vertex $v$ is called a \emph{$k$-vertex}, \emph{$k^+$-vertex}, or \emph{$k^-$-vertex} if $d(v)=k$, $d(v)\ge k$, or $d(v)\le k$, respectively. A set of cardinality exactly $k$ is called a \emph{$k$-set}. We say that a graph $G$ is $k$-regular if $d(v)=k$ for every vertex $v$ in $G$, and we say that $G$ is subcubic if $\Delta(G)\le3$. For a graph $F$, the \emph{complete subdivision} of $F$, denoted by $S(F)$, is the graph obtained from $F$ by subdividing every edge exactly once; that is, each edge $uv\in E(F)$ is replaced by a path $u x_{uv} v$ of length 2, where all vertices $x_{uv}$ are distinct.

Let \(Q\) be a path component of the subgraph of \(G\) induced by its \(2\)-vertices. Exactly two edges of \(G\) have one endpoint in \(V(Q)\) and the other outside \(V(Q)\). If the endpoints of these two edges outside \(Q\) are \(3^+\)-vertices, then \(Q\) is called a \emph{thread}. If \(|V(Q)|=k\), then \(Q\) is called a \emph{\(k\)-thread}; a \emph{\(k^+\)-thread} is a thread containing at least \(k\) vertices. The endpoints outside \(Q\) of these two edges are called the \emph{anchors} of \(Q\), and they are not required to be distinct. The thread \(Q\) is said to be \emph{incident with} each of its anchors. 
For a vertex $v$, let $n_{2^+\text{-thread}}(v)$ denote the number of $2^+$-threads incident with $v$, let $n_{1}(v)$ denote the number of $1$-neighbors of $v$.

For a positive integer $k$, let $[k]=\{1,2,\dots,k\}$. A proper $k$-coloring $\varphi$ of a graph $G$ is a mapping $\varphi:V(G)\to[k]$ satisfying $\varphi(v)\neq\varphi(u)$ for every edge $uv\in E(G)$. We say that a proper coloring $\varphi$ of $G$ is \emph{proper conflict-free} (or simply \textnormal{PCF}) if for every non-isolated vertex $v$ of $G$, there exists a color that appears exactly once in its neighborhood. Such a color is called a \emph{unique color} of $v$. The \emph{proper conflict-free chromatic number} of $G$, denoted $\chi_{\mathrm{pcf}}(G)$, is the minimum number of colors so that $G$ admits a proper conflict-free coloring.

Let $G$ be a graph. A \emph{list assignment} $L$ for $G$ is a mapping that assigns to each vertex $v\in V(G)$ a set $L(v)$ of allowable colors. An \emph{$L$-coloring} is a mapping $\varphi:V(G)\to\bigcup_{v\in V(G)}L(v)$ such that $\varphi(v)\in L(v)$ for each $v$ and $\varphi(u)\neq\varphi(v)$ for every edge $uv\in E(G)$. Given a function $f:V(G)\to\mathbb{N}$, an \emph{$f$-list assignment} is a list assignment $L$ satisfying $|L(v)|\ge f(v)$ for all $v\in V(G)$. The graph $G$ is \emph{$f$-choosable} if it admits an $L$-coloring for every $f$-list assignment $L$. In particular, for a positive integer $k$, $G$ is \emph{$k$-choosable} if it is $f$-choosable with $f(v)\equiv k$ for every vertex $v$ of $G$; and $G$ is \emph{degree-choosable} if it is $f$-choosable with $f(v)=d(v)$ for every vertex $v$ of $G$.

A graph $G$ is called \emph{proper conflict-free $f$-choosable} if for every $f$-list assignment $L$, there exists a proper conflict-free $L$-coloring. Standard variants derived from this framework include:
\begin{itemize}
    \item \emph{proper conflict-free $k$-choosable} if $f(v)\equiv k$;
    \item \emph{proper conflict-free (degree+$k$)-choosable} if $f(v)=d(v)+k$ for some positive integer $k$.
\end{itemize}

In this paper, 
for a graph $G$, let $\kappa_G:V(G)\to\mathbb{N}$ be defined by
\[
\kappa_G(v)=
\begin{cases}
4, & \text{if } d_G(v)=2,\\[4pt]
d_G(v)+1, & \text{if } d_G(v)\neq2.
\end{cases}
\]
We call an $f$-list assignment with $f=\kappa_G$ a \emph{$\kappa_G$-list assignment}. Thus $G$ is proper conflict-free $\kappa_G$-choosable if every list assignment $L$ satisfying $|L(v)|\ge \kappa_G(v)$ for all $v\in V(G)$ admits a proper conflict-free $L$-coloring of $G$.

All undefined graph-theoretic terminology follows the standard reference of Bondy and Murty~\cite{bondy1979graph}.

\subsection{Background and Main Results}

The notion of list coloring was introduced independently by Vizing~\cite{vizing1976vertex} and by Erd\H{o}s, Rubin, and Taylor~\cite{erdos1979choosability} in the late 1970s. As a generalization of Brooks' theorem, degree-choosable graphs were classified by Borodin~\cite{borodin1977criterion} and Erd\H{o}s, Rubin, and Taylor~\cite{erdos1979choosability}:

\begin{theorem}[\cite{borodin1977criterion,erdos1979choosability}]
A connected graph is degree-choosable if and only if it is not a Gallai tree, where a Gallai tree is a graph whose blocks are complete graphs or odd cycles.
\end{theorem}

Proper conflict-free coloring was introduced by Fabrici et al.~\cite{fabrici2023proper}. The study of this coloring in terms of the maximum degree was initiated by Caro, Petru\v{s}evski, and \v{S}krekovski~\cite{caro2023remarks}, who proposed the following Brooks-type conjecture.

\begin{conjecture}[\cite{caro2023remarks}]
\label{con:maximum-degree-pcf}
If \(G\) is a connected graph with \(\Delta(G)\ge3\), then \(\chi_{\mathrm{pcf}}(G)\le\Delta(G)+1\).
\end{conjecture}

This bound is tight already for complete graphs. The best‑known upper bounds are summarized in Table~\ref{tab:pcf-maximum-degree}.

\begin{table}[htbp]
\centering
\caption{Upper bounds for \(\chi_{\mathrm{pcf}}(G)\), where \(\Delta=\Delta(G)\).}
\label{tab:pcf-maximum-degree}
\begin{tabular}{|c|c|c|}
\hline
Range of \(\Delta\) & Upper bound & Reference\\
\hline
\(\Delta=3\) & \(4\) & \cite{liu2013linear,caro2023remarks}\\
\hline
\(4\le\Delta\le749\) & \(2\Delta-1\) & \cite{cho2025brooks}\\
\hline
\(750\le\Delta\le3999\) & \(\frac95\Delta+\sqrt{\Delta}\) & \cite{cranston2024proper}\\
\hline
\(4000\le\Delta\le27999\) & \(\frac53\Delta+\sqrt{\Delta}\) & \cite{cranston2024proper}\\
\hline
\(\Delta\ge28000\) & \(\Delta+600e\log\Delta\) & \cite{chuet2026new}\\
\hline
\end{tabular}
\end{table}

For sparse graphs, Cho et al.~\cite{cho2025proper} obtained the following result.

\begin{theorem}[\cite{cho2025proper}]
\label{thm:cho-pcf}
If \(G\) is a graph with \(\operatorname{mad}(G)<\frac{12}{5}\) and no induced \(C_5\), then \(\chi_{\mathrm{pcf}}(G)\le4\).
\end{theorem}

It is well‑known that every planar graph of girth at least \(g\) satisfies \(\operatorname{mad}(G)<\frac{2g}{g-2}\). Substituting $g=12$, we see the theorem yields an immediate consequence for planar graphs with girth at least 12.
Anderson et al.~\cite{anderson2025forb} further developed the Forb‑Flex Method to obtain better planar‑graph results.
\begin{theorem}[\cite{anderson2025forb}]
If \(G\) is a planar graph with girth at least \(11\), then \(\chi_{\mathrm{pcf}}(G)\le4\).
\end{theorem}

Recently, the combination of proper conflict-free coloring with list coloring has attracted considerable attention. Kashima, \v{S}krekovski, and Xu~\cite{kashima2025results} introduced the notion of proper conflict-free $(\mathrm{degree}+k)$-choosability. A fundamental question raised in this direction is whether there exists an absolute constant $k$ such that every graph is proper conflict-free $(\mathrm{degree}+k)$-choosable~\cite{kashima2026degree}. They proved that
the complete subdivision $S(F)$ is proper conflict-free $(\mathrm{degree}+2)$-choosable for every graph $F$, and every graph of maximum degree at most $4$ is proper conflict-free $(\mathrm{degree}+3)$-choosable. Specifically, for subcubic graphs, they gave the following theorem.

\begin{theorem}[\cite{kashima2025results}]\label{sub}
If $G$ is a connected subcubic graph other than $C_5$, then $G$ is proper conflict-free $(\mathrm{degree}+2)$-choosable.
\end{theorem}

 Wang and Zhang~\cite{wang2025proper} proved the same proper conflict-free $(\mathrm{degree}+2)$‑choosability result for $K_4$-minor-free graphs with $\Delta\le 4$ and outer-$1$-planar graphs with $\Delta\le 4$. They further proved the following result about planar graphs.

\begin{theorem}[\cite{wang2025proper}]\label{t1}
Every planar graph with girth at least $12$ is proper conflict-free $(\mathrm{degree}+2)$-choosable.
\end{theorem}

From the viewpoint of degeneracy, Kashima, \v{S}krekovski, and Xu proved that every $d$-degenerate graph is proper conflict-free $(\mathrm{degree}+d+1)$-choosable and every tree is proper conflict-free $(\mathrm{degree}+1)$-choosable~\cite{kashima2026degeneracy}.

Since every graph with $\operatorname{mad}(G)<k$ has degeneracy less than $k$, we immediately obtain the following observation.

\begin{observation}\label{obs:mad-degree-plus}
Let $k\ge2$ be an integer. If $\operatorname{mad}(G)<k$, then $G$ is proper conflict-free $(\mathrm{degree}+k)$-choosable.
\end{observation}

Kashima, \v{S}krekovski, and Xu also studied proper conflict-free list coloring of some classes of sparse graphs. They 
proved that every connected outerplanar graph other than $C_5$ is proper conflict-free $(\mathrm{degree}+2)$-choosable~\cite{kashima2026proper} and gave the following result and conjecture in~\cite{kashima2026degree}.

\begin{theorem}[\cite{kashima2026degree}]\label{kashima-sparse}
Let $G$ be a graph.
\begin{itemize}
    \item If $\operatorname{mad}(G)<\frac{10}{3}$, then $G$ is proper conflict-free $(\mathrm{degree}+3)$-choosable.
    \item If $G$ is connected, $G\neq C_5$, and $\operatorname{mad}(G)<\frac{18}{7}$, then $G$ is proper conflict-free $(\mathrm{degree}+2)$-choosable.
\end{itemize}
\end{theorem}


\begin{conjecture}[\cite{kashima2026degree}]\label{con1}
Every connected graph other than $C_5$ is proper conflict-free $\kappa_G$-choosable.
\end{conjecture}

 Since  \(\chi_{\mathrm{pcf}}(C_n)=4\) when  \(3\nmid n\) and \(n\neq5\) and \(\chi_{\mathrm{pcf}}(S(K_n))=n=\Delta(S(K_n))+1\) for every \(n\ge3\)~\cite{caro2023remarks}, the function \(\kappa_G\) is sharp. Note that the prescribed list size at a \(1\)-vertex is only \(2\). Consequently, \textbf{Conjecture~\ref{con1}} remains highly nontrivial even for sparse graphs.
 Furthermore, Kashima, \v{S}krekovski, and Xu~\cite{kashima2026proper} constructed infinitely many connected outerplanar graphs that are not proper conflict-free $(\mathrm{degree}+1)$-choosable.   
 Therefore, the exceptional value \(4\) at a \(2\)-vertex and the value \(d_G(v)+1\) at a non-$2$-vertex $v$ are both best possible.

In this paper, we confirm \textbf{Conjecture~\ref{con1}} in two classes of graphs. Our first result concerns graphs with bounded maximum average degree, which readily yield the corresponding result for planar graphs.

\begin{theorem}\label{th1}
Let \(G\) be a connected graph other than \(C_5\). If \(\operatorname{mad}(G)<\frac{12}{5}\), then \(G\) is proper conflict-free \(\kappa_G\)-choosable.
\end{theorem}


\begin{corollary}
Every planar graph with girth at least \(12\) is proper conflict-free \(\kappa_G\)-choosable.
\end{corollary}


Compared with \textbf{Theorem~\ref{t1}} and the second assertion of \textbf{Theorem~\ref{kashima-sparse}}, this replaces the $(\mathrm{degree}+2)$ prescription by the sharper function $\kappa_G$. In particular, the required list size decreases from $3$ to $2$ at $1$-vertices and from $d_G(v)+2$ to $d_G(v)+1$ at every $3^+$-vertex, while the sharp value $4$ at $2$-vertices is retained. For clarity, we summarize the above results on proper conflict-free choosability of connected graphs under bounds on the maximum average degree in Table~\ref{tab:degree-plus-mad}.

\begin{table}[htbp]
\centering
\caption{Known results on proper conflict-free choosability of connected graphs under bounds on the maximum average degree.}
\label{tab:degree-plus-mad}
\begin{tabular}{|c|c|c|}
\hline
Condition & Conclusion & Reference\\
\hline
$\operatorname{mad}(G)<k$, $k\in\mathbb N$, $k\ge4$
& proper conflict-free $(\mathrm{degree}+k)$-choosable
& \textbf{Observation~\ref{obs:mad-degree-plus}}\\
\hline
$\operatorname{mad}(G)<\frac{10}{3}$
& proper conflict-free $(\mathrm{degree}+3)$-choosable
& \textbf{Theorem~\ref{kashima-sparse}}\\
\hline
$\operatorname{mad}(G)<\frac{18}{7}$, $G\neq C_5$
& proper conflict-free $(\mathrm{degree}+2)$-choosable
& \textbf{Theorem~\ref{kashima-sparse}}\\
\hline
$\operatorname{mad}(G)<\frac{12}{5}$, $G\neq C_5$
& proper conflict-free $\kappa_G$-choosable
& \textbf{Theorem~\ref{th1}}\\
\hline
$\operatorname{mad}(G)<2$
& proper conflict-free $(\mathrm{degree}+1)$-choosable
& \cite{kashima2026degeneracy}\\
\hline
\end{tabular}
\end{table}

Our second result concerns subcubic graphs, which improves \textbf{Theorem~\ref{sub}} at both $1$-vertices and $3$-vertices.

\begin{theorem}\label{th2}
Every connected subcubic graph other than \(C_5\) is proper conflict-free \(\kappa_G\)-\allowbreak choosable.
\end{theorem}



The remainder of this paper is organized as follows. In \textbf{Section~\ref{s:auxiliary}}, we introduce the concept of a proper conflict-free representative system and the construction of $r$-leaf-ring graphs, together with several auxiliary lemmas. 
\textbf{Section~\ref{s:structural}} identifies the reducible configurations to be used in subsequent arguments. \textbf{Section~\ref{s2}} gives the discharging‑based proof of \textbf{Theorem~\ref{th1}}. \textbf{Section~\ref{s3}} concludes with the proof of \textbf{Theorem~\ref{th2}}.

\section{Auxiliary Lemmas}\label{s:auxiliary}

Let $X=(x_1,\ldots,x_t)$ be a finite sequence, possibly empty. For a
color $\alpha$, let \(\nu_X(\alpha)=|\{j\in[t]:x_j=\alpha\}|\)
denote the multiplicity of $\alpha$ in $X$, and let \(\operatorname{supp}(X) =\{\alpha:\nu_X(\alpha)\ge1\}\)
denote the set of colors appearing in $X$. In the present sequence
setting, a color $\alpha$ is called a \emph{unique color of $X$} if
$\nu_X(\alpha)=1$. In particular, if $X$ is the sequence of colors
appearing in the neighborhood of a vertex $v$, then a unique color
of $X$ is exactly a unique color of $v$ in the sense defined in
\textbf{Section~\ref{s1}}. For two finite sequences $X$ and $Y$, we
denote their concatenation by $X\circ Y$.

\begin{definition}\label{def:pcf-representative}
Let $\mathcal{A}=(A_1,\ldots,A_k)$ be a sequence of nonempty sets,
and let $B=(b_1,\ldots,b_\ell)$ be a finite sequence, possibly empty.
A sequence $(a_1,\ldots,a_k)$ is called a \textbf{\textnormal{\textbf{PCF}}-representative system} for $\mathcal{A}$ with respect to $B$ if $a_i\in A_i$ for every $i\in[k]$ and the concatenated
sequence
\[
(a_1,\ldots,a_k)\circ B
=(a_1,\ldots,a_k,b_1,\ldots,b_\ell)
\]
has a unique color.
\end{definition}

The following lemma gives an exact criterion, in terms of prescribed
lower bounds on the sizes of the sets, for the guaranteed existence of
a \textnormal{PCF}-representative system.

\begin{lemma}\label{lem:pcf-representative-system}
Let $k$ be a positive integer, let $m_1,\ldots,m_k$ be positive
integers, and let $B=(b_1,\ldots,b_\ell)$ be a finite sequence,
possibly empty. Define $s=|\{\alpha:\nu_B(\alpha)=1\}|$, $r=|\{\alpha:\nu_B(\alpha)\ge2\}|$,
and let $q=|\{i\in[k]:m_i=1\}|$, $D=\max_{i\in[k]}m_i$.
Then the following two statements are equivalent.

\begin{itemize}
    \item[\textnormal{\textbf{(1)}}]
    For every sequence
    $\mathcal{A}=(A_1,\ldots,A_k)$ of nonempty finite sets satisfying
    $|A_i|\ge m_i$ for every $i\in[k]$, there exists a \textnormal{PCF}-representative system for $\mathcal{A}$ with
    respect to $B$.

    \item[\textnormal{\textbf{(2)}}]
    Either $q<s$, or
    \[
    q\ge s
    \quad\text{and}\quad
    D>
    r+s+\left\lfloor\frac{q-s}{2}\right\rfloor.
    \]
\end{itemize}
\end{lemma}

\begin{proof}
Let \(I=\{i\in[k]:m_i=1\}\),
so $|I|=q$. Let $i_1,\ldots,i_q$ be an ordering of the elements of $I$ (with the empty ordering when $q=0$).

We first show that \textbf{(1)}$\implies$\textbf{(2)}, by contrapositive. Assume \textbf{(2)} fails. Then
\[
q\ge s
\quad\text{and}\quad
D\le
r+s+\left\lfloor\frac{q-s}{2}\right\rfloor.
\]
Let \(t=\left\lfloor\frac{q-s}{2}\right\rfloor\).
Denote by $\sigma_1,\ldots,\sigma_s$ the distinct colors appearing
exactly once in $B$. Select $t$ pairwise-distinct new colors
$\tau_1,\ldots,\tau_t$ not contained in $B$. Let \(P=\operatorname{supp}(B) \cup\{\tau_1,\ldots,\tau_t\}\).
Thus \(|P|=r+s+t\ge D\).
In particular, $P\neq\emptyset$, since $D\ge1$.

We assign sets to indices in $I$ as follows.
For each $j\in[s]$, set \(A_{i_j}=\{\sigma_j\}\).
For each $h\in[t]$, set \(A_{i_{s+2h-1}}=A_{i_{s+2h}}=\{\tau_h\}\).
This consumes $s+2t$ indices from $I$. If $q-s$ is even,
then $s+2t=q$, so all indices in $I$ have been assigned. If
$q-s$ is odd, then $s+2t=q-1$; pick an arbitrary
color $\rho\in P$ and set \(A_{i_q}=\{\rho\}\).
For every $i\notin I$, choose $A_i$ to be an arbitrary
$m_i$-subset of $P$. This choice is possible since
$m_i\le D\le|P|$. Consequently, \(|A_i|=m_i\) for every $i\in[k]$.

Take arbitrary representatives $a_i\in A_i$ for every $i\in[k]$. Each color
$\sigma_j$ appears once in $B$ and must appear at least
once among $a_1,\ldots,a_k$, giving its total multiplicity at least two in the concatenation \((a_1,\ldots,a_k)\circ B\). Any color already occurring at least twice within $B$ trivially has multiplicity \(\ge 2\), while each new color $\tau_h$ is guaranteed to show up at least twice among the chosen representatives.
Moreover, all representatives belong to $P$. Hence every color appearing in
$(a_1,\ldots,a_k)\circ B$ appears at least twice. Therefore, this
sequence has no unique color. 

The constructed set‑sequence 
$\mathcal{A}=(A_1,\ldots,A_k)$ therefore possesses no \textnormal{PCF}-representative system with respect to $B$. This shows 
\textbf{(1)} fails, which establishes the contrapositive and yields the implication $(1)\implies(2)$.

We next show that \textbf{(2)}$\implies$\textbf{(1)}. Let
$\mathcal{A}=(A_1,\ldots,A_k)$ be an arbitrary sequence of nonempty
finite sets satisfying \(|A_i|\ge m_i\) for every $i\in[k]$.
For every $i\in I$, choose an arbitrary color
$a_i\in A_i$, and let \(A=(a_{i_1},\ldots,a_{i_q})\)
be the sequence of these chosen colors.

Suppose first that $q<s$. There are $s$ distinct colors appearing
exactly once in $B$, while the sequence $A$ has only $q$ entries.
Consequently, at least one of these $s$ colors does not appear in
$A$. Choose such a color and denote it by $\gamma$. Thus
$\nu_B(\gamma)=1$ and $\nu_A(\gamma)=0$.
For every $i\notin I$, we have $|A_i|\ge m_i\ge2$. We may therefore choose \(a_i\in A_i\setminus\{\gamma\}\).
The color $\gamma$ appears exactly once in
$(a_1,\ldots,a_k)\circ B$. 
Hence $(a_1,\ldots,a_k)$ is a \textnormal{PCF}-representative system for $\mathcal{A}$ with respect to $B$.

It remains to consider the case $q\ge s$ and $D> r+s+\left\lfloor\frac{q-s}{2}\right\rfloor$.
Let \(X=A\circ B\).
If $X$ has a unique color, say $\gamma$, then for every $i\notin I$, choose \(a_i\in A_i\setminus\{\gamma\}\). Such a choice is possible since $|A_i|\ge2$. None of the additional representatives receives color $\gamma$, so $\gamma$ remains a unique color of $(a_1,\ldots,a_k)\circ B$. 
We therefore obtain a required \textnormal{PCF}-representative system. 

We may consequently assume that $X$ has no unique color. Every color
appearing exactly once in $B$ must then appear at least once in
$A$, since otherwise it would remain a unique color of $X$. Thus
at least $s$ entries of $A$ are used by the $s$ unique colors of
$B$.

Furthermore, every color in
$\operatorname{supp}(A)\setminus\operatorname{supp}(B)$ must appear
at least twice in $A$, since it does not appear in $B$ and is not
a unique color of $X$. After reserving at least $s$ entries of
$A$ for the unique colors of $B$, at most $q-s$ entries remain.
It follows that
$
|\operatorname{supp}(A)\setminus\operatorname{supp}(B)|
\le
\left\lfloor\frac{q-s}{2}\right\rfloor.
$
Since $|\operatorname{supp}(B)|=r+s$, we obtain
\[
|\operatorname{supp}(X)|
\le
r+s+\left\lfloor\frac{q-s}{2}\right\rfloor
<D.
\]
Choose $j\in[k]$ such that $m_j=D$. We claim that $j\notin I$, i.e., $D\ge2$.
Otherwise, assume that $D=1$. Then $m_i=1$ for every $i\in[k]$, so
$q=k$. Since $1=D>
r+s+\left\lfloor\frac{q-s}{2}\right\rfloor$, 
the nonnegative integer on the right-hand side must be zero. Hence
$r=s=0$ and $q=1$. In this case, $X=A$ consists of exactly one
color and therefore has a unique color, contrary to our present
assumption. 

Since \(|A_j|\ge D>|\operatorname{supp}(X)|\), there exists a color
\(\alpha\in A_j\setminus\operatorname{supp}(X)\).
Set $a_j=\alpha$. For every
$i\in[k]\setminus(I\cup\{j\})$, we have $|A_i|\ge2$, so choose \(a_i\in A_i\setminus\{\alpha\}\).
The color $\alpha$ does not appear in $X$, and among the remaining
representatives it is used only for $A_j$. Therefore, $\alpha$
appears exactly once in
$(a_1,\ldots,a_k)\circ B$. Hence
$(a_1,\ldots,a_k)$ is a \textnormal{PCF}-representative system for $\mathcal{A}$ with respect to $B$.

From the above arguments, we always find a required \textnormal{PCF}-representative system. This proves \(\textbf{(2)}\Longrightarrow\textbf{(1)}\) and completes the proof.
\end{proof}

In the arguments below, when considering the colors of the neighborhood of a vertex $u$, $\mathcal A$ will usually consist of
the residual allowable-color sets of uncolored neighbors of $u$. These sets are defined after excluding the colors needed for
properness and the prescribed $\varphi^*$-colors which guarantee the existence of a unique color for the relevant vertices other than $u$ under the current partial coloring $\varphi$. Thus a
\textnormal{PCF}-representative system for $\mathcal A$ gives precisely the required choice of colors around $u$, while some color occurs exactly once in $N_G(u)$. In the remainder of the discussion, we shall apply \textbf{Lemma~\ref{lem:pcf-representative-system}} to verify the existence of the desired coloring, rather than analyzing concrete coloring configurations.

In the proof of \textbf{Theorem~\ref{th1}}, we also use the discharging method. To avoid troubles arising from 1‑vertices in the minimal counterexample $G$, we construct an auxiliary graph of $G$ preserving the constraint imposed by $\operatorname{mad}(G)$. 

Fix an integer $r\ge4$. Let $G$ be a graph with no component isomorphic to $K_2$, set
\[
X=\{x\in V(G):d_G(x)=1\},
\qquad
K=G-X,
\]
and, for every $v\in V(K)$ having a neighbor in $X$, write
\(N_G(v)\cap X=\{x_1,\ldots,x_{n_1(v)}\}\).

\begin{definition}\label{def:leaf-ring}
The graph $H_r(G)$ is called the \emph{$r$-leaf-ring graph of
$G$} if it is obtained from $G$ by performing the following operation
for every $v\in V(K)$ with $n_1(v)\ge1$.
\begin{itemize}
    \item If $n_1(v)\ge2$, then for each $i\in[n_1(v)]$ add an
    $x_i$--$x_{i+1}$ path $P_{v,i}$ of length $r$, where
    $x_{n_1(v)+1}=x_1$,  \(V(P_{v,i})\cap V(G)=\{x_i,x_{i+1}\}\) and  $P_{v,1},\ldots,P_{v,n_1(v)}$ are internally
    vertex-disjoint.

    \item If $n_1(v)=1$, then add a cycle $C_v$ of length $r$
    satisfying
    \(V(C_v)\cap V(G)=\{x_1\}\).
\end{itemize}
No vertex of $V(H_r(G))\setminus V(G)$ belongs to two distinct
paths or cycles added newly in the construction.
\end{definition}

For $v\in V(K)$ with $\ell=n_1(v)\ge1$, let $R_v$ denote the
subgraph consisting of $v$, its $\ell$ leaf-neighbors in $G$, and
all vertices and edges added for those leaf-neighbors. We call $R_v$
the \emph{$r$-leaf-ring at $v$}. In either case,
$R_v-v$ is a cycle; we denote it by $C_v$. $C_v$ has $r\ell$ vertices
and $r\ell$ edges. 
Moreover, every vertex of $X$ has degree $3$ in $H_r(G)$, every
new vertex has degree $2$, and every vertex of $K$ has the same
degree in $H_r(G)$ as in $G$.

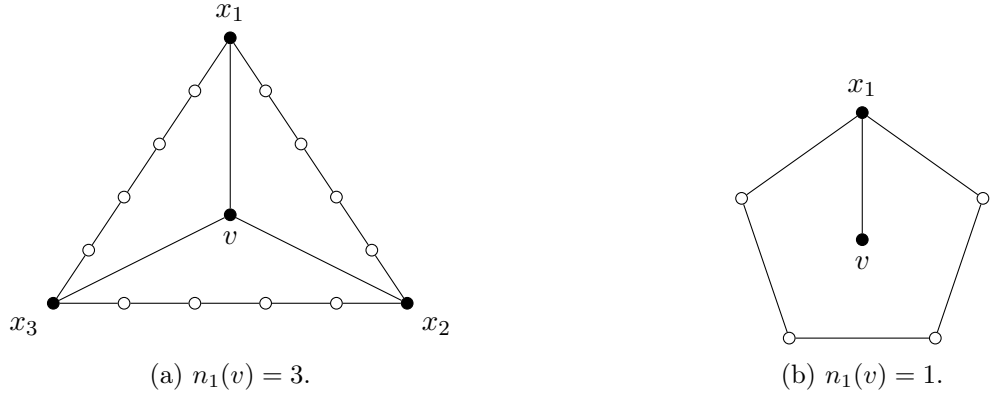
\begin{figure}[htbp]
\centering
\begin{subfigure}[t]{0.61\textwidth}
\centering
\begin{tikzpicture}[scale=0.78]
\tikzset{
old vertex/.style={draw,shape=circle,fill=black,inner sep=1.55pt},
new vertex/.style={draw,shape=circle,fill=white,inner sep=1.55pt}
}
\node[old vertex,label=below:$v$] (v) at (0,0) {};
\node[old vertex,label=above:$x_1$] (x1) at (0,3.0) {};
\node[old vertex,label=below right:$x_2$] (x2) at (3.0,-1.5) {};
\node[old vertex,label=below left:$x_3$] (x3) at (-3.0,-1.5) {};
\draw (v)--(x1);
\draw (v)--(x2);
\draw (v)--(x3);

\foreach \t/\name in {0.2/a12,0.4/b12,0.6/c12,0.8/d12}
  \node[new vertex] (\name) at ($(x1)!\t!(x2)$) {};
\draw (x1)--(a12)--(b12)--(c12)--(d12)--(x2);

\foreach \t/\name in {0.2/a23,0.4/b23,0.6/c23,0.8/d23}
  \node[new vertex] (\name) at ($(x2)!\t!(x3)$) {};
\draw (x2)--(a23)--(b23)--(c23)--(d23)--(x3);

\foreach \t/\name in {0.2/a31,0.4/b31,0.6/c31,0.8/d31}
  \node[new vertex] (\name) at ($(x3)!\t!(x1)$) {};
\draw (x3)--(a31)--(b31)--(c31)--(d31)--(x1);
\end{tikzpicture}
\caption{$n_1(v)=3$.}
\end{subfigure}
\hfill
\begin{subfigure}[t]{0.34\textwidth}
\centering
\begin{tikzpicture}[scale=0.84]
\tikzset{
old vertex/.style={draw,shape=circle,fill=black,inner sep=1.55pt},
new vertex/.style={draw,shape=circle,fill=white,inner sep=1.55pt}
}
\node[old vertex,label=below:$v$] (v) at (0,0.0) {};
\node[old vertex,label=above:$x_1$] (x1) at (0,2.0) {};
\node[new vertex] (a) at (1.9,0.65) {};
\node[new vertex] (b) at (1.15,-1.55) {};
\node[new vertex] (c) at (-1.15,-1.55) {};
\node[new vertex] (d) at (-1.9,0.65) {};
\draw (v)--(x1);
\draw (x1)--(a)--(b)--(c)--(d)--(x1);
\end{tikzpicture}
\caption{$n_1(v)=1$.}
\end{subfigure}
\caption{Examples of a \(5\)-leaf-ring at \(v\). Solid vertices belong to
\(G\), and hollow vertices are newly added.}
\label{fig:leaf-ring}
\end{figure}

\begin{lemma}\label{lem:leaf-ring-mad}
Let $r\ge4$, and let $G$ be a graph with no component isomorphic to
$K_2$. If \(\operatorname{mad}(G)<2+\frac{2}{r}\), then
\(\operatorname{mad}(H_r(G))<2+\frac{2}{r}\).
\end{lemma}

\begin{proof}
Let $H=H_r(G)$. For a graph $F$, define \(\Phi_r(F)=(r+1)|V(F)|-r|E(F)|\).
For every nonempty graph $F$,
\[
\frac{2|E(F)|}{|V(F)|}<2+\frac{2}{r}
\quad\Longleftrightarrow\quad
\Phi_r(F)>0.
\]
It is therefore enough to prove that \(\Phi_r(H[Y])>0\) for every nonempty set $Y\subseteq V(H)$.

Fix such a set $Y$, and set \(S=Y\cap V(K)\).
If $S\neq\emptyset$, then $K[S]=G[S]$, and the assumption
$\operatorname{mad}(G)<2+2/r$ gives \(\Phi_r(K[S])>0\).
We now determine the contribution of the selected vertices from each
leaf-ring relative to this core subgraph $G[S]$. 
The vertex sets of distinct leaf-rings are disjoint outside
$V(K)$, and no edge in $E(H)\setminus E(K)$ belongs to two distinct
leaf-rings. Therefore, their contributions relative to $K[S]$ may be
summed independently.

Fix a leaf-ring $R_v$ with $Y_v=V(C_v)\cap Y\neq \emptyset$, write \(N_G(v)\cap X=\{x_1,\ldots,x_\ell\}\).
We distinguish four
cases.

\medskip
\noindent\textbf{Case 1: $v\notin S$ and $Y_v\neq V(C_v)$.}

Since the whole cycle $C_v$ is not selected, the graph $C_v[Y_v]$
is a disjoint union of paths, where an isolated vertex is regarded as
a path. A nonempty path component with $t$ vertices has $t-1$ edges
and hence contributes
\[
(r+1)t-r(t-1)=t+r>0
\]
to $\Phi_r(H[Y])$. 

\medskip
\noindent\textbf{Case 2: $v\notin S$ and $Y_v=V(C_v)$.}

The selected part is the entire cycle $C_v$, and none of the edges
$vx_i$ belongs to $H[Y]$. Since $C_v$ has $r\ell$ vertices and
$r\ell$ edges, its contribution is
\[
(r+1)r\ell-r(r\ell)=r\ell>0.
\]

\medskip
\noindent\textbf{Case 3: $v\in S$ and $Y_v\neq V(C_v)$.}

Again, $C_v[Y_v]$ is a disjoint union of paths. Let $Q$ be a
nonempty path component, let $t=|V(Q)|$, and let $q$ be the number
of vertices among $x_1,\ldots,x_\ell$ that belong to $Q$. In
addition to the $t-1$ edges of $Q$, the graph $H[Y]$ contains the
$q$ edges joining these vertices to the already selected vertex $v$.
Thus the increase 
relative to $K[S]$ is
\[
(r+1)t-r\bigl((t-1)+q\bigr)=t-r(q-1).
\]
If $q=0$, this value is $t+r>0$. If $q\ge1$, then consecutive
vertices of $\{x_1,\ldots,x_\ell\}$ on $C_v$ are at distance $r$.
Consequently,
\[
t\ge r(q-1)+1,
\]
and the increase is at least $1$. 

\medskip
\noindent\textbf{Case 4: $v\in S$ and $Y_v=V(C_v)$.}

In this case, \(\Phi_r(K[S])>0\) and the whole leaf-ring outside $v$ is selected. Such additional part contributes
$r\ell$ vertices, $r\ell$ edges of $C_v$, and the $\ell$ edges
$vx_i$. Hence its increase relative to $K[S]$ is
\[
(r+1)r\ell-r(r\ell+\ell)=0.
\]

Cases~1--4 show that every leaf-ring adds a nonnegative amount. From the analysis of all cases, we always get \(\Phi_r(H[Y])>0\) and complete the proof.
\end{proof}


We shall also use the following characterization of the exceptional
$4$-list assignments of $C_5$.

\begin{lemma}[\cite{kashima2026proper}]\label{lem:c5-characterization}
Let $C=v_0v_1v_2v_3v_4v_0$ be a $5$-cycle, and let $L$ be a
$4$-list assignment of $C$. Then $C$ is not \textnormal{PCF}
$L$-colorable if and only if \(L(v_0)=L(v_1)=L(v_2)=L(v_3)=L(v_4)\)
and $|L(v_0)|=4$.
\end{lemma}

We now introduce the coloring convention used. We always delete a nonempty
vertex set $S$ from a minimal counterexample $G$ to our results. By inductive assumption, any component $F$ of \(G-S\) admits a \textnormal{PCF} $L$-coloring, unless it satisfies the condition described in \textbf{Lemma~\ref{lem:c5-characterization}}. We call such $F$ an \emph{exceptional $C_5$-component}, i.e., all five lists on $F$ are the same $4$-set $A_F$, every vertex of $\partial_SF$ is a $3$-vertex of $G$ and has exactly one neighbor in $S$, where \(\partial_SF=\{x\in V(F):N_G(x)\cap S\neq\emptyset\}\).


For an exceptional $C_5$-component $F$ with the color set $ A_F$, fix a prescribed vertex $x_F\in\partial_SF$ and a prescribed color $\alpha_F\in A_F$. Write \(F=x_Fv_1v_2v_3v_4x_F\) and $ A_F= \{\alpha_F,\beta_F,\gamma_F,\delta_F\}$. We give an  \emph{admissible $L$-coloring} $\varphi$ of $F$ with the \emph{distinguished vertex} $x_F$ and  the \emph{distinguished color} $\alpha_F$ as follows:

\begin{equation*}
\varphi(v_1)=\varphi(v_4)=\alpha_F,\qquad
\varphi(x_F)=\beta_F,\qquad
\varphi(v_2)=\gamma_F,\qquad
\varphi(v_3)=\delta_F.
\end{equation*}

\begin{definition}\label{def:admissible}
Let $\emptyset\neq S\subseteq V(G)$. A proper $L$-coloring
$\varphi$ of $G-S$ is called \emph{admissible} if the following
conditions hold for every component $F$ of $G-S$.
\begin{itemize}
    \item If $F$ is not an exceptional $C_5$-component, then
    $\varphi|_F$ is a \textnormal{PCF} $L$-coloring of $F$.

    \item If $F$ is an exceptional $C_5$-component, then 
    $\varphi|_F$ is an admissible  $L$-coloring of $F$ for some pre-chosen distinguished vertex $x_F$ and distinguished color $\alpha_F$.
\end{itemize}
\end{definition}

For an admissible coloring $\varphi$ of $G-S$ and a
non-isolated vertex $v$ of $G-S$, we define $\varphi^*(v)$ as follows. If $v$ is not the
distinguished vertex of an exceptional $C_5$-component, then
$\varphi^*(v)$ is chosen to be  an arbitrary color occurring exactly once in
$N_{G-S}(v)$. If $v=x_F$ is the distinguished vertex of an exceptional $C_5$-component, set \(\varphi^*(x_F)=\alpha_F\).
In the latter case, the two neighbors of $x_F$ in $F$ receive both 
$\alpha_F$, while $x_F$ has a unique neighbor $u$ in $S$. When coloring the vertices in $S$, say $u$, we will avoid $\varphi(v)$ and $\varphi^*(v)$ for each neighbor $v$ of $u$. The operation will guarantee the existence of unique colors of all already colored vertices, where the distinguished vertex of an exceptional $C_5$-component obtains a unique color from its unique neighbor in $S$. This ultimately yields the \emph{desired coloring}, namely a \textnormal{PCF} $L$-coloring of $G$.



\section{Structural Lemmas}\label{s:structural}

We first isolate a reducibility
property of a minimal counterexample that will be used in both proofs of our main results.

\begin{lemma}\label{lem:leaf-neighbor}
Let $G$ be a connected graph, and let $L$ be a list assignment with
$|L(v)|=\kappa_G(v)$ for every $v\in V(G)$. Suppose that $G$ has no
\textnormal{PCF} $L$-coloring and that, for every nonempty proper set
$S\subsetneq V(G)$, every component $F$ of $G-S$ with
$F\not\cong C_5$ admits a \textnormal{PCF} $L|_F$-coloring. Then every $1$-vertex of $G$ has a
$4^+$-neighbor.
\end{lemma}

\begin{proof}
We prove this by contradiction. Suppose that $x$ is a $1$-vertex with its unique neighbor $u$ and $d_G(u)\le3$. Note that $|L(x)|=2$. If $d_G(u)=1$, then connectedness gives $G=K_2$,
which is immediately \textnormal{PCF} $L$-colorable, a contradiction with the hypothesis. 

Suppose that $d_G(u)=2$. Then $|L(u)|=4$.
Let $w$ be the other neighbor of $u$. Set $S=\{x,u\}$.
Take an admissible coloring $\varphi$ of $G-S$ where $w$ is the distinguished vertex when $G-S$ is an exceptional $C_5$; such a coloring exists by the hypothesis and
\textbf{Lemma~\ref{lem:c5-characterization}}. 
Choose
\(\varphi(x)\in L(x)\setminus\{\varphi(w)\}\).
If $w$ is non-isolated in $G-S$, choose
\[
\varphi(u)\in
L(u)\setminus\{\varphi(x),\varphi(w),\varphi^*(w)\};
\]
otherwise choose
\(\varphi(u)\in L(u)\setminus\{\varphi(x),\varphi(w)\}\).
Obviously, the resulting coloring is a \textnormal{PCF}
$L$-coloring of $G$, a contradiction.


It remains to consider $d_G(u)=3$. Let $G'=G-x$. Then $G'$ is
connected, and $L|_{G'}$ is a $\kappa_{G'}$-list assignment because
$u$ changes from degree $3$ to degree $2$, for which the prescribed
list size is still $4$. We take an admissible coloring $\varphi$ of $G'$ where $u$ is the distinguished vertex and the distinguished color $\alpha_{G'}\in L(u)\backslash L(x)$  when $G'$ is an exceptional $C_5$; 
such a coloring exists by the hypothesis and
\textbf{Lemma~\ref{lem:c5-characterization}}.  

 Choose \(\varphi(x)\in L(x)\setminus\{\varphi(u)\}\). If $G'$ is not an exceptional $C_5$, the two neighbors of $u$ receive distinct colors. So $u$ has at least one unique color left after coloring $x$.  If $G'$ is an exceptional $C_5$, then $\varphi(x)\neq \alpha_{G'}$ for the choice of $\alpha_{G'}$, and  $\varphi(x)$ becomes the unique color of $u$. Therefore the resulting coloring is a \textnormal{PCF}
$L$-coloring of $G$, a contradiction.
\end{proof}

For the remainder of this section, let $G$ be a counterexample to
\textbf{Theorem~\ref{th1}} of minimum order. Choose a
$\kappa_G$-list assignment $L_0$ for which $G$ has no \textnormal{PCF}
$L_0$-coloring, and, for each $v\in V(G)$, choose
$L(v)\subseteq L_0(v)$ with $|L(v)|=\kappa_G(v)$. Thus $G$ has no
\textnormal{PCF} $L$-coloring and
\[
|L(v)|=
\begin{cases}
2, & \text{if }d_G(v)=1,\\
4, & \text{if }d_G(v)=2,\\
d_G(v)+1, & \text{if }d_G(v)\ge3.
\end{cases}
\]

Obviously, $G$ is connected, $G\ncong C_4$, $G\ncong C_5$ and  $|V(G)|\ge 4$. We record the form of minimality used below. Let
$\emptyset\neq S\subset V(G)$, and let $F$ be a component of $G-S$.
For every $v\in V(F)$,
\(|L(v)|=\kappa_G(v)\ge\kappa_F(v)\).
Moreover,
$\operatorname{mad}(F)\le\operatorname{mad}(G)<12/5$. Hence every
component $F\neq C_5$ has a \textnormal{PCF} $L$-coloring by minimality, while a
$C_5$-component is handled by
\textbf{Lemma~\ref{lem:c5-characterization}} and
\textbf{Definition~\ref{def:admissible}}. In particular, the
hypotheses of \textbf{Lemma~\ref{lem:leaf-neighbor}} hold.

Set
\[
X=\{x\in V(G):d_G(x)=1\},
\qquad
K=G-X,
\qquad
H=H_5(G).
\]
Then $\delta(H)\ge2$ and \(\operatorname{mad}(H)<\frac{12}{5}\) by \textbf{Lemma~\ref{lem:leaf-ring-mad}}.

We call the vertices of $V(H)\setminus V(K)$ \emph{ring vertices}
and the vertices of $V(K)$ \emph{core vertices}.
For a core vertex
$v$, let
\[
n_{\mathrm{ring}}(v)
=|N_H(v)\cap(V(H)\setminus V(K))|.
\]
Every core vertex has the same degree and the same core neighbors in
$H$ as in $G$, and its ring neighbors are exactly the $1$-neighbors
it had in $G$. Hence $d_H(v)=d_G(v)$ and $n_{\mathrm{ring}}(v)
=|\{x\in N_G(v):d_G(x)=1\}|$ 
for every core vertex $v$.

\begin{lemma}\label{lem:ring-structure}
The following statements hold in $H$.
\begin{itemize}
    \item[\textnormal{\textbf{(i)}}]
    Every ring vertex has degree $2$ or $3$. A ring $3$-vertex has
    exactly one core neighbor and two ring neighbors, whereas a ring
    $2$-vertex has two ring neighbors and no core neighbor.

    \item[\textnormal{\textbf{(ii)}}]
    The unique core neighbor of every ring $3$-vertex is a core
    $4^+$-vertex.

    \item[\textnormal{\textbf{(iii)}}]
    Every ring $2$-vertex lies on a $4$-thread.

    \item[\textnormal{\textbf{(iv)}}]
    Let $x$ be a ring $3$-vertex. If $x$ is the sole anchor of a
    $4$-thread, then this is the only thread incident with $x$.
    Otherwise, $x$ is an anchor of exactly two $4$-threads, each
    having two distinct ring $3$-vertices as anchors.

    \item[\textnormal{\textbf{(v)}}]
    Every thread incident with a core vertex contains only core
    $2$-vertices.
\end{itemize}
\end{lemma}

\begin{proof}
These statements 
follow directly from \textbf{Definition~\ref{def:leaf-ring}} and \textbf{Lemma~\ref{lem:leaf-neighbor}}.
\end{proof}

We shall derive structural restrictions on the core vertices and apply the discharging method on $H$. The resulting contradiction shows that the auxiliary
graph $H=H_5(G)$ associated with a minimal counterexample $G$ cannot exist, and hence the minimal
counterexample $G$ itself cannot exist. Whenever a set $S$ is deleted
from $G$, the minimality of $G$ together with
\textbf{Lemma~\ref{lem:c5-characterization}} gives an admissible
coloring of each component of $G-S$ as described in
\textbf{Section~\ref{s:auxiliary}}.
Whenever such an admissible coloring $\varphi$ is fixed, we use
the notation $\varphi^*(v)$ for a non-isolated vertex
$v\in V(G-S)$ exactly as specified in
\textbf{Section~\ref{s:auxiliary}}.

The structural lemmas below are stated for $H$, but their reducibility
proofs are carried out in $G$. Indeed, each ring $3$-neighbor of a
core vertex is an original $1$-vertex of $G$, while every thread incident
with a core vertex is inherited unchanged from $G$. Thus an assumed
configuration in $H$ yields the corresponding configuration in $G$,
where we delete the indicated vertices and extend an admissible
coloring. 


\begin{lemma}\label{lem:structure-1}
The following statements hold in $H$.
\begin{itemize}
    \item[\textnormal{\textbf{(i)}}]
    No core $3$-vertex is incident with a $2^+$-thread.

    \item[\textnormal{\textbf{(ii)}}]
    No thread incident with a core vertex is a $4^+$-thread.

    \item[\textnormal{\textbf{(iii)}}]
    Every core $2$-vertex lies on a thread whose two anchors are
    distinct core $3^+$-vertices.
\end{itemize}
\end{lemma}

\begin{proof}
We prove the assertions in order.

\medskip
\noindent\textbf{(i)}
Suppose that a core $3$-vertex $u$ is incident with a
$2^+$-thread. This thread is inherited from $G$. Let $v,w$ be its
first two $2$-vertices starting from $u$, and let $\xi$ be the
neighbor of $w$ distinct from $v$. Put $S=\{v,w\}$. 
We will choose an admissible coloring $\varphi$ for each component of $G-S$ and extend it to be a \textnormal{PCF} coloring of $G$ which is also denoted by $\varphi$. 



A separate argument is required only when the component $F$ of $G-S$ containing $u$ is an exceptional $C_5$-component. 
In this case $\xi\neq u$; otherwise $u$ has degree at most one in $G-S$, so its component cannot be isomorphic to $C_5$. Choose $u$ as the distinguished vertex of $F$. Color all other components of $G-S$ admissibly. 
Choose 
\[
\varphi(w)\in
\begin{cases}
L(w)\setminus\{\varphi(\xi),\varphi^*(\xi)\}, & \xi\notin F\\
L(w), & \xi\in F
\end{cases}
\]
and then choose
\[
\varphi(v)\in
\begin{cases}
L(v)\setminus\{\varphi(\xi),\varphi(w)\}, & \xi\notin F\\
L(v)\setminus\{\varphi(w)\}, & \xi\in F
\end{cases}.
\]
Let $A_F$ denote the common $4$-set assigned to the vertices of the
exceptional $C_5$-component $F$. Choose
\[
\beta\in A_F\setminus\{\varphi(v),\varphi(w)\}
\quad\text{and}\quad
\alpha\in A_F\setminus\{\beta,\varphi(v),\varphi(w)\}.
\]
Color $F$ so that $u$ receives $\beta$, its two cycle-neighbors both
receive $\alpha$, and the remaining two cycle vertices can be properly colored by the
two colors in $A_F\setminus\{\alpha,\beta\}$. It is easy to verify that the resulting coloring is a required 
 \textnormal{PCF} $L$-coloring of $G$, a
contradiction.

Now suppose that the component $F$ of $G-S$ containing $u$ is not an exceptional $C_5$-component. If $\xi=u$, then $u$ has degree $1$ in $G-S$ and $G-S$ has a \textnormal{PCF} $L$-coloring $\varphi$.
Choose $\varphi(w)\in L(w)\setminus \{\varphi(\xi)\}$ and $\varphi(v)\in L(v)\setminus \{\varphi(\xi),\varphi(w)\}$. 
It is easy to verify that the resulting coloring is a \textnormal{PCF} $L$-coloring of $G$, a
contradiction. So $\xi\neq u$ and $d_{G-S}(u)=2$.
We can choose an admissible coloring $\varphi$ of $G-S$ where the two neighbors of $u$ have different colors. Here $\varphi^*(\xi)$ is well defined. 
Choose \[ \varphi(w)\in L(w)\setminus \{\varphi(\xi),\varphi(u),\varphi^*(\xi)\} \] and then \[ \varphi(v)\in L(v)\setminus \{\varphi(\xi),\varphi(w),\varphi(u)\}. \]
Both choices are possible because $|L(v)|=|L(w)|=4$. The coloring
is proper, and  all vertices in $\{u,v,w,\xi\}$ have their own unique colors. Therefore the resulting coloring is a \textnormal{PCF} $L$-coloring of $G$, a contradiction. 


\medskip
\noindent\textbf{(ii)}
Suppose that a thread incident with a core vertex contains at least
four $2$-vertices. Let $v_1,v_2,v_3,v_4$ be four consecutive
$2$-vertices starting from a core anchor, and let $u$ and $w$ be
the neighbors of $v_1$ and $v_4$, respectively, outside
$\{v_1,v_2,v_3,v_4\}$. Put
$S=\{v_1,v_2,v_3,v_4\}$ and take an admissible $L$-coloring
$\varphi$ of $G-S$.

By \textbf{Lemma~\ref{lem:leaf-neighbor}}, $d_G(u)\ge 2$ and $d_G(w)\ge 2$. Furthermore,  $d_G(u)=d_G(w) \ge 3$ when $u=w$, since $G\ncong C_5$. So any endpoint $x\in\{u,w\}$ is non-isolated in $G-S$ and $\varphi^*(x)$ is well defined.

Set
\[
\begin{aligned}
L^*(v_1)&=L(v_1)\setminus
\{\varphi(u),\varphi^*(u)\},&
L^*(v_2)&=L(v_2)\setminus\{\varphi(u)\},\\
L^*(v_3)&=L(v_3)\setminus\{\varphi(w)\},&
L^*(v_4)&=L(v_4)\setminus
\{\varphi(w),\varphi^*(w)\}.
\end{aligned}
\]
Then $|L^*(v_1)|,|L^*(v_4)|\ge2$ and
$|L^*(v_2)|,|L^*(v_3)|\ge3$. Choose a $2$-element subset
$M_1\subseteq L^*(v_1)$, and then choose successively
\[
\varphi(v_2)\in L^*(v_2)\setminus M_1,
\qquad
\varphi(v_4)\in L^*(v_4)\setminus\{\varphi(v_2)\},
\]
\[
\varphi(v_3)\in
L^*(v_3)\setminus\{\varphi(v_2),\varphi(v_4)\},
\qquad
\varphi(v_1)\in
M_1\setminus\{\varphi(v_3)\}.
\]
All choices are possible and guarantee that 
\[
\varphi(v_1)\neq\varphi(u),\quad
\varphi(v_2)\neq\varphi(v_1),\quad
\varphi(v_3)\neq\varphi(v_2),\quad
\varphi(v_4)\neq\varphi(w),
\]
and 
\[
\varphi(v_2)\neq\varphi(u),\quad
\varphi(v_1)\neq\varphi(v_3),\quad
\varphi(v_2)\neq\varphi(v_4),\quad
\varphi(v_3)\neq\varphi(w).
\]
Furthermore, for each vertex $q\in\{u,w\}$,  $\varphi^*(q)$ is still its unique color. 
 Hence the resulting coloring is a \textnormal{PCF} $L$-coloring of $G$, a
contradiction.


\medskip
\noindent\textbf{(iii)}
Suppose that a core $2$-vertex does not lie on a thread with two
distinct core $3^+$-anchors. Let $Q$ be its component of the
subgraph of $G$ induced by the $2$-vertices. The graph $Q$ is a path
or a cycle. Every cycle other than $C_5$ is \textnormal{PCF} $\kappa_G$-choosable by \textbf{Theorem~\ref{sub}}. Hence, together with (ii), $Q$ is a path and $|V(Q)|\le 3$.



Write $Q=v_1\cdots v_t$, where $1\le t\le3$. By maximality of
$Q$ and \textbf{Lemma~\ref{lem:leaf-neighbor}}, each vertex outside $Q$ adjacent to an endpoint of $Q$ has degree at least $3$. By the assumption, both endpoints of $Q$ are adjacent to the same vertex $u$. Since $G$ is
simple, $t\in\{2,3\}$. Put $S=\{v_1,\ldots,v_t\}$. The vertex $u$
has two neighbors in $S$ and therefore cannot be the distinguished
vertex of an exceptional $C_5$-component. Thus an admissible coloring $\varphi$ of $G-S$ may be chosen so that
$u$ has a unique color $\varphi^*(u)$. 
If $t=2$, choose
\[
\varphi(v_2)\in
L(v_2)\setminus\{\varphi(u),\varphi^*(u)\},
\]
\[
\varphi(v_1)\in
L(v_1)\setminus
\{\varphi(u),\varphi^*(u),\varphi(v_2)\}.
\]
If $t=3$, choose
\[
\varphi(v_3)\in
L(v_3)\setminus\{\varphi(u),\varphi^*(u)\},
\]
\[
\varphi(v_1)\in
L(v_1)\setminus
\{\varphi(u),\varphi^*(u),\varphi(v_3)\},
\]
\[
\varphi(v_2)\in
L(v_2)\setminus
\{\varphi(u),\varphi(v_1),\varphi(v_3)\}.
\]
The resulting coloring is 
a \textnormal{PCF}
$L$-coloring of $G$, a contradiction.
\end{proof}

We now begin to use \textnormal{PCF}-representative systems. 
Based on the requirements of a proper $L$-coloring and that all vertices except one vertex possess unique colors, we define the available color set for each vertex in the so‑called vertex‑deleted set $S$. 
The remaining work is to find a \textnormal{PCF}-representative system.  If no such system
exists, \textbf{Lemma~\ref{lem:pcf-representative-system}} forces many
of the available-color sets to be singletons, and a double-counting
argument yields a contradiction. We shall apply this scheme repeatedly.

\begin{lemma}\label{lem:structure-3}
Let $k\ge4$, and let $u$ be a core $k$-vertex in $H$. If
\(n_{\mathrm{ring}}(u)=k-2\), then $u$ has a core $3^+$-neighbor.
\end{lemma}

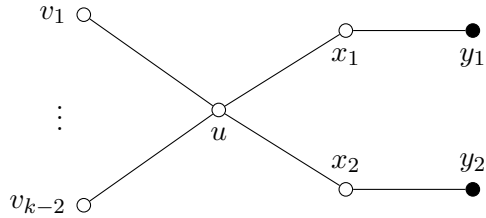
\begin{figure}[htbp]
\centering
\begin{tikzpicture}[scale=1.05]
\tikzset{
deleted vertex/.style={draw,shape=circle,fill=white,inner sep=1.8pt},
retained vertex/.style={draw,shape=circle,fill=black,inner sep=1.8pt}
}
\node[deleted vertex,label=below:$u$] (u) at (0,0) {};
\node[deleted vertex,label=below:$x_1$] (x1) at (1.6,1.0) {};
\node[deleted vertex,label=above:$x_2$] (x2) at (1.6,-1.0) {};
\node[retained vertex,label=below:$y_1$] (y1) at (3.2,1.0) {};
\node[retained vertex,label=above:$y_2$] (y2) at (3.2,-1.0) {};
\node[deleted vertex,label=left:$v_1$] (v1) at (-1.7,1.2) {};
\node[draw=none] at (-2.0,0) {$\vdots$};
\node[deleted vertex,label=left:$v_{k-2}$] (vk) at (-1.7,-1.2) {};
\draw (u)--(x1)--(y1);
\draw (u)--(x2)--(y2);
\draw (u)--(v1);
\draw (u)--(vk);
\end{tikzpicture}
\caption{Configuration excluded in \textbf{Lemma~\ref{lem:structure-3}}. Hollow
vertices are deleted, and solid vertices remain colored.}
\label{fig:structure-3}
\end{figure}

\begin{proof}
Suppose, to the contrary, that $u$ is a core $k$-vertex in
$H$, where $k\ge4$, with $n_{\mathrm{ring}}(u)=k-2$ and no core
$3^+$-neighbor. In $G$, the $k-2$ ring neighbors of $u$ correspond
to $1$-neighbors $v_1,\ldots,v_{k-2}$. The two remaining neighbors
of $u$ are $2$-vertices, denote them by $x_1$ and $x_2$.

By \textbf{Lemma~\ref{lem:structure-1} (iii)}, $x_1x_2\notin E(G)$.
For $i\in[2]$,
let $y_i$ be the neighbor of $x_i$ distinct from $u$. Put
\(S=\{u,x_1,x_2,v_1,\ldots,v_{k-2}\}\), and take an admissible
$L$-coloring $\varphi$ of $G-S$. 
By \textbf{Lemma~\ref{lem:leaf-neighbor}}, $d_G(y_1)\ge 2$ and  $d_G(y_2)\ge 2$. Furthermore, by \textbf{Lemma~\ref{lem:structure-1} (iii)},  $d_G(y_1)=d_G(y_2)\ge 3$ when $y_1=y_2$. Thus,  for $i\in[2]$,  each $y_i$ is non-isolated in $G-S$, and  
$\varphi^*(y_i)$ is defined.

Define \(C=L(u)\setminus\{\varphi(y_1),\varphi(y_2)\}\). Since
$|L(u)|=k+1$, we have $|C|\ge k-1$. For $i\in[2]$, define
\(P_i=L(x_i)\setminus\{\varphi(y_i),\varphi^*(y_i)\}\). Then
$|P_i|\ge2$.

Fix $\alpha\in C$. For $j\in[k-2]$, define
\(A_\alpha(v_j)=L(v_j)\setminus\{\alpha\}\), and for $i\in[2]$,
define \(A_\alpha(x_i)=P_i\setminus\{\alpha\}\). All these sets
are nonempty. The choice of these sets follows the requirements for a proper $L$-coloring of $G$ and ensures that every vertex other than $u$ has a unique color.

We claim that, 
for any $\alpha\in C$, 
\[
\mathcal{A}_\alpha=
\bigl(
A_\alpha(v_1),\ldots,A_\alpha(v_{k-2}),
A_\alpha(x_1),A_\alpha(x_2)
\bigr)
\]
has no \textnormal{PCF}-representative system with respect to
the empty sequence. Otherwise, let \((a_1,\ldots,\allowbreak a_{k-2},\allowbreak b_1,b_2)\)
be such a system and extend $\varphi$ by setting
\[
\varphi(u)=\alpha,\qquad
\varphi(v_j)=a_j\quad(j\in[k-2]),\qquad
\varphi(x_i)=b_i\quad(i\in[2]).
\]
$u$ has a unique color now. Hence the resulting coloring is 
a \textnormal{PCF} $L$-coloring of $G$, a contradiction. 

For the fixed color $\alpha$, let \(q_\alpha= \left| \left\{ z\in N_G(u):|A_\alpha(z)|=1 \right\} \right|\)
and \(D_\alpha=\max_{z\in N_G(u)}|A_\alpha(z)|\).
Applying \textbf{Lemma~\ref{lem:pcf-representative-system}} with the
lower bounds $m_z=|A_\alpha(z)|$ and with the empty sequence $B$ yields $D_\alpha\le r+s+\left\lfloor\frac{q_\alpha-s}{2}\right\rfloor$, where $r=s=0$. Therefore, $q_\alpha=k\ge 4$ when $D_\alpha=1$, and  $q_\alpha\ge2D_\alpha\ge4$ when $D_\alpha\ge 2$. 


We now double-count the pairs $(\alpha,z)$ satisfying
$\alpha\in C$, $z\in N_G(u)$, and $|A_\alpha(z)|=1$. The preceding
argument gives at least $4|C|\ge4(k-1)$ such pairs.

For each $v_j$, the set
$A_\alpha(v_j)=L(v_j)\setminus\{\alpha\}$ is a singleton only when
$\alpha\in L(v_j)$, so $v_j$ occurs in at most two such pairs. For
$x_i$, the set $A_\alpha(x_i)=P_i\setminus\{\alpha\}$ can be a
singleton only if $|P_i|=2$ and $\alpha\in P_i$, so $x_i$ also
occurs in at most two such pairs. Since $u$ has $k$ neighbors, the
total number of pairs is at most $2k$. Consequently,
\[
4(k-1)\le2k,
\]
which is impossible for $k\ge4$. This contradiction proves that $u$ has a core
$3^+$-neighbor.
\end{proof}

\begin{lemma}\label{lem:structure-4}
For every integer $k\ge2$, every core $k$-vertex $u$ in $H$
satisfies \(n_{\mathrm{ring}}(u)+\allowbreak n_{2^+\text{-thread}}(u)\allowbreak\le k-2\).
\end{lemma}

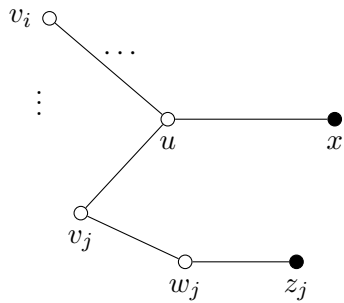
\begin{figure}[htbp]
\centering
\begin{tikzpicture}[scale=0.92]
\tikzset{
deleted vertex/.style={draw,shape=circle,fill=white,inner sep=1.8pt},
retained vertex/.style={draw,shape=circle,fill=black,inner sep=1.8pt}
}
\node[deleted vertex,label=below:$u$] (u) at (0,0) {};
\node[retained vertex,label=below:$x$] (x) at (2.4,0) {};
\node[deleted vertex,label=left:$v_i$] (vi) at (-1.7,1.45) {};
\node[deleted vertex,label=below:$v_j$] (vj) at (-1.25,-1.35) {};
\node[deleted vertex,label=below:$w_j$] (wj) at (0.25,-2.05) {};
\node[retained vertex,label=below:$z_j$] (zj) at (1.85,-2.05) {};
\node[draw=none] at (-1.85,0.35) {$\vdots$};
\node[draw=none] at (-0.65,0.95) {$\cdots$};
\draw (u)--(x);
\draw (u)--(vi);
\draw (u)--(vj)--(wj)--(zj);
\end{tikzpicture}
\caption{Configuration excluded in \textbf{Lemma~\ref{lem:structure-4}}. Each selected
incidence at $u$ is either an edge $uv_i$ or a $2^+$-thread beginning
with $v_jw_j$. Hollow vertices are deleted, and solid vertices remain colored.}
\label{fig:structure-4}
\end{figure}

\begin{proof}
If $k=2$, then $u$ is not an anchor of a thread and has no ring
neighbor by \textbf{Lemma~\ref{lem:ring-structure}}\textbf{(ii)}. If $k=3$, then
\textbf{Lemma~\ref{lem:structure-1}}\textbf{(i)} excludes a
$2^+$-thread and \textbf{Lemma~\ref{lem:ring-structure}}\textbf{(ii)} excludes a
ring neighbor. Thus the assertion holds for $k\le3$.

Assume $k\ge4$ and, to the contrary, that
\(n_{\mathrm{ring}}(u)+n_{2^+\text{-thread}}(u)\ge k-1\).
Choose $k-1$ distinct incidences at $u$, each represented either by
an edge from $u$ to a ring $3$-vertex or by a $2^+$-thread incident
with $u$. Let $I_1$ and
$I_{\mathrm{thread}}$ be the corresponding index sets, so that
$|I_1\cup I_{\mathrm{thread}}|=k-1$, and put
$h=|I_1|$.

For $i\in I_1$, let $v_i$ be the $1$-neighbor of $u$
in $G$ corresponding to the selected ring neighbor. For
$i\in I_{\mathrm{thread}}$, let $v_i,w_i$ be the first two vertices
of the selected thread starting from $u$, and let $z_i$ be the
neighbor of $w_i$ distinct from $v_i$. Let $x$ be the remaining
neighbor of $u$. Thus \(N_G(u)=\{v_1,\ldots,v_{k-1},x\}\).
Set
\[
S=\{u\}\cup\{v_i:i\in I_1\}
\cup\{v_i,w_i:i\in I_{\mathrm{thread}}\},
\]
and take an admissible $L$-coloring $\varphi$ of $G-S$.

We first verify that every $z_i$ is non-isolated in $G-S$. If the
selected thread is a $3$-thread, then $z_i$ is its third $2$-vertex,
and its neighbor toward the other anchor remains outside $S$. Now
suppose that the selected thread is a $2$-thread, so $z_i$ is its
other anchor. Assume for a contradiction that a vertex $z$ arising in this way is
isolated in $G-S$. Let $m$ be the number of selected
$2$-threads with anchors $u$ and $z$, and put
$\varepsilon=1$ if $uz\in E(G)$ and $\varepsilon=0$ otherwise.
All neighbors of $z$ then lie in $S$, and the degree constraints on
the internal $2$-vertices give
\(d_G(z)=m+\varepsilon\).
The subgraph induced by $u,z$ and the $2m$ internal vertices of these
threads has at least $3m+\varepsilon$ edges. Since
$\operatorname{mad}(G)<12/5$,
\[
\frac{2(3m+\varepsilon)}{2m+2}<\frac{12}{5},
\]
which is equivalent to $6m+10\varepsilon<24$. Hence $m\le3$ when
$\varepsilon=0$, and $m\le2$ when $\varepsilon=1$. In either case, $d_G(z)=m+\varepsilon\le3$. Since $z$ is an
anchor of a thread, $d_G(z)\ge3$, and hence $d_G(z)=3$. This
contradicts \textbf{Lemma~\ref{lem:structure-1}}\textbf{(i)},
because $z$ is incident with a $2^+$-thread. Thus
every $z_i$ is non-isolated, and $\varphi^*(z_i)$ is defined.

Put $\beta=\varphi(x)$. If $x$ is non-isolated in $G-S$, define
\(C=L(u)\setminus\{\beta,\varphi^*(x)\}\);
otherwise define $C=L(u)\setminus\{\beta\}$. In either case,
$|C|\ge k-1$.

Fix $\alpha\in C$. For every $i\in I_{\mathrm{thread}}$, choose
\[
c_i^\alpha\in
L(w_i)\setminus
\{\alpha,\varphi(z_i),\varphi^*(z_i)\}.
\]
This is possible because $|L(w_i)|=4$. When $z_i=z_j$ for distinct indices $i,j$, the corresponding
colors $c_i^\alpha$ and $c_j^\alpha$ are different from the same
color $\varphi^*(z_i)$.

For $i\in[k-1]$, define
\[
A_i^\alpha=
\begin{cases}
L(v_i)\setminus\{\alpha\},
& \text{if }i\in I_1,\\[4pt]
L(v_i)\setminus
\{\alpha,c_i^\alpha,\varphi(z_i)\},
& \text{if }i\in I_{\mathrm{thread}}.
\end{cases}
\]
Each set $A_i^\alpha$ is nonempty. Let $B=(\beta)$.

We claim that
$\mathcal A_\alpha=(A_1^\alpha,\ldots,A_{k-1}^\alpha)$ admits no
\textnormal{PCF}-representative system with respect to $B$. Suppose, for contradiction, that such a system exists. Set 
$\varphi(u)=\alpha$, $\varphi(w_i)=c_i^\alpha$ for each $i\in I_{\mathrm{thread}}$, and assign the representatives to $v_1,\ldots,v_{k-1}$ accordingly.  
The resulting coloring is a \textnormal{PCF} $L$-coloring of $G$, a contradiction.

For $\alpha\in C$, set
\[
q_\alpha=|\{i\in[k-1]:|A_i^\alpha|=1\}|,
\qquad
D_\alpha=\max_{i\in[k-1]}|A_i^\alpha|.
\]
Applying
\textbf{Lemma~\ref{lem:pcf-representative-system}} to $\mathcal A_\alpha \circ B$  yields that  $D_\alpha\le r+s+\left\lfloor\frac{q_\alpha-s}{2}\right\rfloor$, where $s=1$ and $r=0$.

Let
\[
C_1=\{\alpha\in C:D_\alpha=1\},\quad
C_2=\{\alpha\in C:D_\alpha=2\},\quad
C_3=\{\alpha\in C:D_\alpha\ge3\}.
\]
If $\alpha\in C_1$, then $q_\alpha=k-1$; if
$\alpha\in C_2$, then $q_\alpha\ge3$; and if
$\alpha\in C_3$, then $q_\alpha\ge5$. Therefore
\begin{equation}\label{eq:structure-4-singleton-lower}
\sum_{\alpha\in C}q_\alpha
\ge3|C|+(k-4)|C_1|+2|C_3|.
\end{equation}

For each $i\in I_1$, the set $A_i^\alpha$ is a
singleton for at most two colors $\alpha\in C$. For each
$i\in I_{\mathrm{thread}}$, the set $A_i^\alpha$ is a singleton for
at most three colors, because $|A_i^\alpha|=1$ implies
$\alpha\in L(v_i)\setminus\{\varphi(z_i)\}$. Hence
\begin{equation}\label{eq:structure-4-singleton-upper}
\sum_{\alpha\in C}q_\alpha
\le2h+3(k-1-h)=3(k-1)-h.
\end{equation}
Since $|C|\ge k-1$, \eqref{eq:structure-4-singleton-lower} and \eqref{eq:structure-4-singleton-upper} imply
\[
3(k-1)+(k-4)|C_1|+2|C_3|\le3(k-1)-h.
\]
Thus $h=0$, $C_3=\emptyset$, and, when $k\ge5$,
$C_1=\emptyset$. Equality holds throughout, so $|C|=k-1$,
$q_\alpha=3$ for every $\alpha\in C$. Furthermore, for each $i\in I_1$, the set $A_i^\alpha$ is a singleton for exactly two colors $\alpha\in C$,  and for every $i\in I_{\mathrm{thread}}$, the set
$A_i^\alpha$ is a singleton for exactly three colors
$\alpha\in C$. 
Hence, $ L(v_i)\subseteq C$ for each $i\in I_1$, and  $L(v_i)\setminus\{\varphi(z_i)\}\subseteq  C$ for each $i\in I_{\mathrm{thread}}$. That is,
\begin{equation}\label{eq:structure-4-available-set-containment}
A_i^\alpha\subseteq C
\end{equation}
holds for each $i\in [k-1]$.  


Fix $\alpha\in C$, color $u$ with $\alpha$, use the colors $c_i^\alpha$ on the vertices $w_i$, and choose arbitrary representatives $a_i\in A_i^\alpha$ for $v_i$. By \eqref{eq:structure-4-available-set-containment}, every $v_i$ receives a color in $C$, while $\beta=\varphi(x)\notin C$. Thus $\beta$ appears exactly once in $N_G(u)$.
Thus the resulting coloring is a \textnormal{PCF} $L$-coloring of $G$, a contradiction that completes the proof.
\end{proof}

\begin{lemmacorollary}\label{cor:structure-4}
Let $u$ be a core $k$-vertex in $H$, where $k\ge2$. Then:
\begin{itemize}
    \item[\textnormal{\textbf{(i)}}]
    $n_{\mathrm{ring}}(u)\le k-2$.

    \item[\textnormal{\textbf{(ii)}}]
    If $n_{\mathrm{ring}}(u)=k-3$, then
    $n_{2^+\text{-thread}}(u)\le1$.

    \item[\textnormal{\textbf{(iii)}}]
    If $n_{\mathrm{ring}}(u)=k-2$, then
    $n_{2^+\text{-thread}}(u)=0$.
\end{itemize}
\end{lemmacorollary}

By \textbf{Corollary~\ref{cor:structure-4}}, every core
$3^+$-vertex is of exactly one of the following two types. A core
$3^+$-vertex $v$ is called \emph{good} if
$n_{\mathrm{ring}}(v)\le d_H(v)-3$, and \emph{bad} if
$n_{\mathrm{ring}}(v)=d_H(v)-2$.

\begin{lemma}\label{lem:structure-5}
No bad core vertex is incident with a $1$-thread.
\end{lemma}

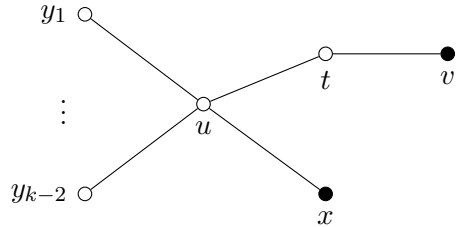
\begin{figure}[htbp]
\centering
\begin{tikzpicture}[scale=0.95]
\tikzset{
deleted vertex/.style={draw,shape=circle,fill=white,inner sep=1.8pt},
retained vertex/.style={draw,shape=circle,fill=black,inner sep=1.8pt}
}
\node[deleted vertex,label=below:$u$] (u) at (0,0) {};
\node[deleted vertex,label=below:$t$] (t) at (1.7,0.7) {};
\node[retained vertex,label=below:$v$] (v) at (3.4,0.7) {};
\node[retained vertex,label=below:$x$] (x) at (1.7,-1.25) {};
\node[deleted vertex,label=left:$y_1$] (y1) at (-1.65,1.25) {};
\node[draw=none] at (-1.95,0) {$\vdots$};
\node[deleted vertex,label=left:$y_{k-2}$] (yk) at (-1.65,-1.25) {};
\draw (u)--(t)--(v);
\draw (u)--(x);
\draw (u)--(y1);
\draw (u)--(yk);
\end{tikzpicture}
\caption{Configuration excluded in \textbf{Lemma~\ref{lem:structure-5}}. The vertex
\(u\) is a bad core vertex, \(utv\) is a \(1\)-thread, and \(x\) and
\(v\) may coincide. Hollow vertices are deleted.}
\label{fig:structure-5}
\end{figure}

\begin{proof}
Suppose, to the contrary, that a bad core $k$-vertex $u$ is
incident with a $1$-thread. In $G$, the $k-2$ ring neighbors of $u$
correspond to $1$-neighbors $y_1,\ldots,y_{k-2}$, and the inherited
$1$-thread is $utv$, where $t$ is its $2$-vertex and $v$ is its
other anchor.
Let \(x\) be the remaining neighbor of \(u\). Thus
\(N_G(u)=\{y_1,\ldots,y_{k-2},t,x\}\). 
Since $u$ has a ring neighbor,
\textbf{Lemma~\ref{lem:ring-structure}}\textbf{(ii)} gives $k\ge4$. Moreover,
\textbf{Lemma~\ref{lem:structure-3}} gives a core $3^+$-neighbor of
$u$. Since $t$ is a core $2$-vertex, $x$ is a core $3^+$-vertex.
Notice that $x=v$ is possible.

Set \(S=\{u,t,y_1,\ldots,y_{k-2}\}\),
and take an admissible \(L\)-coloring \(\varphi\) of \(G-S\).
Both \(x\) and \(v\) are non-isolated in \(G-S\). 
Hence \(\varphi^*(x)\) and \(\varphi^*(v)\) are defined. 
Put \(\beta=\varphi(x), \eta=\varphi(v), \theta=\varphi^*(v)\),
and define \(C=L(u)\setminus\{\beta,\eta,\varphi^*(x)\}\).
Repeated colors are deleted only once. Since \(|L(u)|=k+1\),
\begin{equation}\label{eq:structure-5-C-size}
|C|\ge k-2.
\end{equation}

Fix \(\alpha\in C\). For \(i\in[k-2]\), define
\(A_\alpha(y_i)=L(y_i)\setminus\{\alpha\}\), and define
\(A_\alpha(t)=L(t)\setminus\{\alpha,\eta,\theta\}\). All these
sets are nonempty. Let
\(\mathcal A_\alpha=(A_\alpha(y_1),\ldots,A_\alpha(y_{k-2}),A_\alpha(t))\)
and let \(B=(\beta)\).

We claim that $\mathcal A_\alpha$ has no \textnormal{PCF}-representative
system with respect to $B$. Otherwise, set $\varphi(u)=\alpha$ and assign the representatives to  $y_1,\ldots,y_{k-2},t$ accordingly.
Thus the resulting coloring is a \textnormal{PCF} $L$-coloring of $G$, a contradiction. 

For \(\alpha\in C\), put
\[
q_\alpha=
\left|
\left\{
z\in\{y_1,\ldots,y_{k-2},t\}:|A_\alpha(z)|=1
\right\}
\right|
\]
and let \(D_\alpha\) be the maximum size of a set in
\(\mathcal A_\alpha\).
Applying
\textbf{Lemma~\ref{lem:pcf-representative-system}} to $\mathcal A_\alpha \circ B$  yields that  $D_\alpha\le r+s+\left\lfloor\frac{q_\alpha-s}{2}\right\rfloor$, where $s=1$ and $r=0$.
If \(D_\alpha=1\), then $q_\alpha=k-1\ge3$. If
\(D_\alpha=2\), then \(q_\alpha\ge3\). If \(D_\alpha=3\), then 
\(q_\alpha\ge5\). If \(D_\alpha= 4\), then 
\(q_\alpha\ge7\).

For \(j\in[4]\), let \(C_j=\{\alpha\in C:D_\alpha=j\}\). Hence
\begin{equation}\label{eq:structure-5-singleton-lower}
\sum_{\alpha\in C}q_\alpha
\ge3|C|+2|C_3|+4|C_4|.
\end{equation}

For each $1$-vertex $y_i$, the set \(A_\alpha(y_i)\) is a singleton for
at most two colors \(\alpha\in C\). If \(|A_\alpha(t)|=1\), then
\(\alpha,\eta,\theta\) are pairwise distinct colors of \(L(t)\),
so \(\alpha\in L(t)\setminus\{\eta,\theta\}\). Hence
\(A_\alpha(t)\) is also a singleton for at most two colors \(\alpha\in C\). Therefore
\begin{equation}\label{eq:structure-5-singleton-upper}
\sum_{\alpha\in C}q_\alpha
\le2(k-2)+2=2(k-1).
\end{equation}
Combining \eqref{eq:structure-5-C-size}--\eqref{eq:structure-5-singleton-upper}, we obtain
\begin{equation*}
k-4+2|C_3|+4|C_4|\le0.
\end{equation*}
Since \(k\ge4\), equality holds in all three inequalities. Therefore,
\(k=4\), \(C_3=C_4=\emptyset\), and \(|C|=2\). Furthermore, \(q_\alpha=3\) for each \(\alpha\in C\), i.e., all the sets \(A_\alpha(y_1), A_\alpha(y_2),  A_\alpha(t)\) are singletons. Write \(C=\{c_1,c_2\}\).
It follows that \(L(y_1)=L(y_2)=C\) and \(L(t)=C\cup\{\eta,\theta\}\).

Now set
\(\varphi(u)=c_1\) and
\(\varphi(y_1)=\varphi(y_2)=\varphi(t)=c_2\).
Since \(\beta\notin C\), the color \(\beta=\varphi(x)\) appears exactly once in \(N_G(u)\). 
This gives a \textnormal{PCF} \(L\)-coloring of \(G\), a contradiction.
\end{proof}

The following lemma shows the two-color flexibility at a bad core vertex. It will be used in \textbf{Lemmas~\ref{lem:structure-6},~\ref{lem:one-thread-good-neighbor}, and~\ref{lem:no-one-thread-bad-neighbor}}.

\begin{lemma}\label{lem:bad-two-color-extension}
Let $b$ be a bad core $d$-vertex in $H$. In $G$, let
$p_1,\ldots,p_{d-2}$ be the $1$-neighbors of $b$ corresponding to
its ring neighbors, and let $u$ and $h$ be its other two neighbors.
Let $S\subseteq V(G)$ such that $h\notin S$ and $\{u,b,p_1,\ldots,p_{d-2}\}\subseteq S$  and let $\varphi$ be an
admissible $L$-coloring of $G-S$.

Then there exist distinct colors $\gamma_1,\gamma_2\in L(b)$ and a
color $\lambda$ such that, for each $i\in[2]$,
\[
\gamma_i\notin 
\begin{cases}
\{\varphi(h),\varphi^*(h)\}, & \text{if $h$ is non-isolated in $G-S$}, \\
\{\varphi(h)\}, & \text{if $h$ is isolated in $G-S$}.
\end{cases}
\]
Moreover, for each $i\in[2]$, there exist colors
\[
a^{(i)}_j\in L(p_j)\setminus\{\gamma_i\}
\qquad (j\in[d-2])
\]
such that
\[
\nu_{(a^{(i)}_1,\ldots,a^{(i)}_{d-2},\varphi(h))}(\lambda)=1.
\]
\end{lemma}

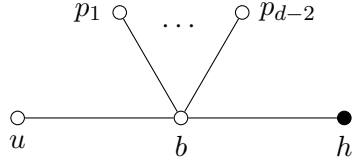
\begin{figure}[htbp]
\centering
\begin{tikzpicture}[scale=0.9]
\tikzset{
deleted vertex/.style={draw,shape=circle,fill=white,inner sep=1.8pt},
retained vertex/.style={draw,shape=circle,fill=black,inner sep=1.8pt}
}
\node[deleted vertex,label=below:$b$] (b) at (0,0) {};
\node[deleted vertex,label=below:$u$] (u) at (-2.4,0) {};
\node[retained vertex,label=below:$h$] (h) at (2.4,0) {};
\node[deleted vertex,label=left:$p_1$] (p1) at (-0.9,1.55) {};
\node[draw=none] at (0,1.35) {$\cdots$};
\node[deleted vertex,label=right:$p_{d-2}$] (pd) at (0.9,1.55) {};
\draw (u)--(b)--(h);
\draw (b)--(p1);
\draw (b)--(pd);
\end{tikzpicture}
\caption{Configuration in \textbf{Lemma~\ref{lem:bad-two-color-extension}}.
The vertex $h$ is colored, while $b$ and its $1$-neighbors are uncolored.}
\label{fig:bad-two-color-extension}
\end{figure}

\begin{proof}
Since $b$ is bad,
\textbf{Lemma~\ref{lem:ring-structure}}\textbf{(ii)} gives $d\ge4$.
Put $\delta=\varphi(h)$, and let
\[
C=
\begin{cases}
L(b)\setminus\{\delta,\varphi^*(h)\},
&\text{if $h$ is non-isolated in $G-S$},\\[4pt]
L(b)\setminus\{\delta\},
&\text{otherwise}.
\end{cases}
\]
Since $|L(b)|=d+1$, we have
\begin{equation}\label{eq:bad-two-color-C-size}
|C|\ge d-1.
\end{equation}

Let
\(I=\{j\in[d-2]:\delta\in L(p_j)\}\).
For each $j\in I$, let $\eta_j$ be the unique color in
$L(p_j)\setminus\{\delta\}$, and put
\(T=\{\eta_j:j\in I\}\).
Since every $p_j$ has a $2$-list,
\begin{equation}\label{eq:bad-two-color-T-size}
|T|\le |I|\le d-2.
\end{equation}

Suppose first that $|T|\le d-3$. By \eqref{eq:bad-two-color-C-size}, choose distinct
\(\gamma_1,\gamma_2\in C\setminus T\)
and set $\lambda=\delta$. For $i\in[2]$, define
$a_j^{(i)}=\eta_j$ when $j\in I$, and choose
\(a_j^{(i)}\in L(p_j)\setminus\{\gamma_i\}\)
when $j\notin I$. These choices are possible, and no
$a_j^{(i)}$ equals $\delta$. Hence $\delta$ occurs exactly once in
\((a^{(i)}_1,\ldots,a^{(i)}_{d-2},\delta)\).

It remains to consider $|T|=d-2$. Equality in \eqref{eq:bad-two-color-T-size} implies that
$I=[d-2]$ and that the colors $\eta_1,\ldots,\eta_{d-2}$ are
pairwise distinct. Choose
\(\gamma_1\in C\setminus T\),
choose $\lambda\in T$, and let $j_0$ be the unique index such that
$\eta_{j_0}=\lambda$. Since $|C|\ge d-1\ge3$, we may further choose
\(\gamma_2\in C\setminus\{\gamma_1,\lambda\}\).
For either $i\in[2]$, set
\[
a_{j_0}^{(i)}=\lambda,
\qquad
a_j^{(i)}=\delta\quad(j\neq j_0).
\]
Because $\gamma_1\notin T$ and $\gamma_2\neq\lambda$, all these
colors belong to the required sets
$L(p_j)\setminus\{\gamma_i\}$. Moreover, $\lambda$ occurs exactly
once in \((a^{(i)}_1,\ldots,a^{(i)}_{d-2},\delta)\).
\end{proof}


\begin{lemma}\label{lem:structure-6}
Let $u$ be a bad core vertex, and let $v$ be a core $3^+$-vertex
adjacent to $u$. Then $v$ is good.
\end{lemma}

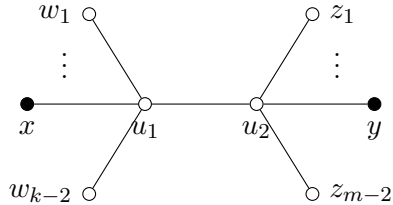
\begin{figure}[htbp]
\centering
\begin{tikzpicture}[scale=0.82]
\tikzset{
deleted vertex/.style={draw,shape=circle,fill=white,inner sep=1.7pt},
retained vertex/.style={draw,shape=circle,fill=black,inner sep=1.7pt}
}
\node[deleted vertex,label=below:$u_1$] (u1) at (-0.9,0) {};
\node[deleted vertex,label=below:$u_2$] (u2) at (0.9,0) {};
\node[retained vertex,label=below:$x$] (x) at (-2.8,0) {};
\node[retained vertex,label=below:$y$] (y) at (2.8,0) {};
\node[deleted vertex,label=left:$w_1$] (w1) at (-1.8,1.45) {};
\node[draw=none] at (-2.2,0.75) {$\vdots$};
\node[deleted vertex,label=left:$w_{k-2}$] (wk) at (-1.8,-1.45) {};
\node[deleted vertex,label=right:$z_1$] (z1) at (1.8,1.45) {};
\node[draw=none] at (2.2,0.75) {$\vdots$};
\node[deleted vertex,label=right:$z_{m-2}$] (zl) at (1.8,-1.45) {};
\draw (x)--(u1)--(u2)--(y);
\draw (u1)--(w1);
\draw (u1)--(wk);
\draw (u2)--(z1);
\draw (u2)--(zl);
\end{tikzpicture}
\caption{Configuration excluded in \textbf{Lemma~\ref{lem:structure-6}}. The two
hollow central vertices are bad core vertices.}
\label{fig:structure-6}
\end{figure}

\begin{proof}
Suppose, to the contrary, that two bad core vertices $u_1$ and
$u_2$, of degrees $k$ and $m$, respectively, are adjacent. Each has
a ring neighbor, so \textbf{Lemma~\ref{lem:ring-structure}}\textbf{(ii)} gives
$k,m\ge4$. In $G$, let $w_1,\ldots,w_{k-2}$ be the $1$-neighbors
of $u_1$ corresponding to its ring neighbors, and let $x$ be its
remaining neighbor besides $u_2$. Similarly, let
$z_1,\ldots,z_{m-2}$ be the $1$-neighbors of $u_2$ corresponding to
its ring neighbors, and let $y$ be its remaining neighbor besides
$u_1$. Set
\[
S=\{u_1,u_2,w_1,\ldots,w_{k-2},z_1,\ldots,z_{m-2}\},
\]
and take an admissible $L$-coloring $\varphi$ of $G-S$.

Put $\delta=\varphi(x)$. If $x$ is non-isolated in $G-S$, define
$C=L(u_1)\setminus\{\delta,\varphi^*(x)\}$; otherwise define
$C=L(u_1)\setminus\{\delta\}$. In either case, $|C|\ge k-1$.

Apply \textbf{Lemma~\ref{lem:bad-two-color-extension}} to the bad
vertex $u_1$, its $1$-neighbors $w_1,\ldots,w_{k-2}$, and the
already colored neighbor $x$. We obtain distinct colors
$\beta,\gamma\in C$ and a color $\alpha$ such that, for either
$c\in\{\beta,\gamma\}$, the vertices $w_1,\ldots,w_{k-2}$ can be
colored properly against $u_1$ colored with $c$, and $\alpha$
appears exactly once in
\[
\bigl(\varphi(w_1),\ldots,\varphi(w_{k-2}),\delta\bigr).
\]

Put $\eta=\varphi(y)$. If $y$ is non-isolated in $G-S$, define
$F=L(u_2)\setminus\{\alpha,\eta,\varphi^*(y)\}$; otherwise define
$F=L(u_2)\setminus\{\alpha,\eta\}$. Since $|L(u_2)|=m+1$,
$|F|\ge m-2$. For each $\theta\in F$, define
$A_\theta(u_1)=\{\beta,\gamma\}\setminus\{\theta\}$ and
$A_\theta(z_i)=L(z_i)\setminus\{\theta\}$ for $i\in[m-2]$. Let
$B=(\eta)$.

Suppose that, for some $\theta\in F$, the sequence
\[
\mathcal A_\theta=
\bigl(A_\theta(u_1),A_\theta(z_1),\ldots,A_\theta(z_{m-2})\bigr)
\]
has a \textnormal{PCF}-representative system with respect to $B$. Use the chosen
representative on $u_1$, the remaining representatives on the
$z_i$, and then apply the corresponding coloring described in 
\textbf{Lemma~\ref{lem:bad-two-color-extension}} in the neighborhood of $u_1$. 
Finally set $\varphi(u_2)=\theta$. Note that the color $\alpha$ remains unique in the neighborhood of $u_1$.
Hence the extension of $\varphi$ is a \textnormal{PCF} $L$-coloring of $G$, a
contradiction. Hence any $\theta\in F$ admits no such system.

  Let $q_\theta$
be the number of singleton sets in $\mathcal A_\theta$ and let
$D_\theta$ be the maximum size of a set in $\mathcal A_\theta$.
There are $m-1$ sets, each of size
at most $2$. 
Applying
\textbf{Lemma~\ref{lem:pcf-representative-system}} to $\mathcal A_\theta \circ B$  yields that  $D_\theta\le r+s+\left\lfloor\frac{q_\theta-s}{2}\right\rfloor$, where $s=1$ and $r=0$. 
This implies that $q_\theta\ge3$:
If $D_\theta=1$, then immediately $q_\theta =m-1\ge3$; if $D_\theta=2$, failure of the criterion gives $q_\theta \ge3$.

Double-count the pairs $(\theta,A)$ with $\theta\in F$,
$A\in\mathcal A_\theta$, and $|A|=1$. There are at least $3|F|$
such pairs. The set $A_\theta(u_1)$ is a singleton for at most the
two colors in $\{\beta,\gamma\}$, and each $A_\theta(z_i)$ is a singleton
for at most two colors in $L(z_i)$.
Thus
\[
3|F|\le2+2(m-2)=2(m-1).
\]
Since $|F|\ge m-2$, we obtain $3(m-2)\le2(m-1)$, impossible when
$m\ge5$.

It remains to consider $m=4$. Equality must hold throughout, so
$F=\{\beta,\gamma\}$ and $L(z_1)=L(z_2)=F$. In particular,
$\beta\in F\subseteq L(u_2)\setminus\{\alpha,\eta\}$, and hence
$\beta\neq\alpha$. Choose
$\varphi(u_2)=\beta$, $\varphi(u_1)=\gamma$, and
$\varphi(z_1)=\varphi(z_2)=\gamma$, and use the coloring of
$w_1,\ldots,w_{k-2}$ supplied by \textbf{Lemma~\ref{lem:bad-two-color-extension}} corresponding to $\gamma$. Since
$\eta\notin F$, the color $\eta$ appears exactly once in $N_G(u_2)$,
while $\alpha$ remains unique in $N_G(u_1)$. 
If both $x$ and $y$ are non-isolated in $G-S$, their fixed unique colors remain.  If one of $x,y$ is isolated in $G-S$, then $x=y$ and this vertex
is a $2$-vertex with neighbors $u_1,u_2$. Since
$\varphi(u_1)=\gamma\neq\beta=\varphi(u_2)$, the two colors in its
neighborhood are distinct. Hence the resulting coloring is a \textnormal{PCF} $L$-coloring of
$G$, a contradiction.

Thus two bad core vertices cannot be adjacent. 
\end{proof}



\begin{lemma}\label{lem:thread-reduction}
Let $u$ be a good core $k$-vertex satisfying
\(n_{\mathrm{ring}}(u)=k-3\),
and suppose that $u$ is incident with a $2^+$-thread $Q$. If $u$
is incident with a $1$-thread, then $k=4$ and $Q$ is a $2$-thread.
\end{lemma}

\begin{figure}[htbp]
\centering
\begin{tikzpicture}[scale=0.76]
\tikzset{
deleted vertex/.style={draw,shape=circle,fill=white,inner sep=1.7pt},
retained vertex/.style={draw,shape=circle,fill=black,inner sep=1.7pt}
}
\node[deleted vertex,label=above:$u$] (u) at (0,0) {};
\node[deleted vertex,label=above:$v$] (v) at (1.3,1.15) {};
\node[deleted vertex,label=above:$w$] (w) at (2.6,1.15) {};
\node[deleted vertex,label=above:$z$] (z) at (3.9,1.15) {};
\node[retained vertex,label=above:$r$] (r) at (5.2,1.15) {};
\node[deleted vertex,label=below:$a$] (a) at (1.35,-1.0) {};
\node[retained vertex,label=below:$p$] (p) at (2.7,-1.0) {};
\node[retained vertex,label=below:$x$] (x) at (0,-2.0) {};
\node[deleted vertex,label=left:$y_1$] (y1) at (-1.55,1.15) {};
\node[draw=none] at (-1.85,0) {$\vdots$};
\node[deleted vertex,label=left:$y_{k-3}$] (yk) at (-1.55,-1.15) {};
\draw (u)--(v)--(w)--(z)--(r);
\draw (u)--(a)--(p);
\draw (u)--(x);
\draw (u)--(y1);
\draw (u)--(yk);
\end{tikzpicture}
\caption{Configuration in \textbf{Lemma~\ref{lem:thread-reduction}} with
$Q=uvwzr$ a $2^+$-thread and $d_G(v)=d_G(w)=2$ ($d_G(z)=2$ if $Q$ is a
$3$-thread).
}
\label{fig:thread-reduction}
\end{figure}
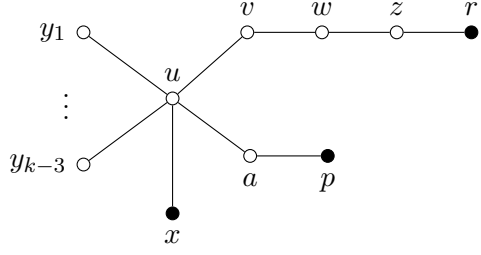

\begin{proof}
By \textbf{Lemma~\ref{lem:structure-1}}\textbf{(i)}, $k\ge4$.
By \textbf{Corollary~\ref{cor:structure-4}}\textbf{(ii)}, $Q$ is
the unique $2^+$-thread incident with $u$, and
\textbf{Lemma~\ref{lem:structure-1}}\textbf{(ii)} implies that $Q$
is a $2$-thread or a $3$-thread. Let
$y_1,\ldots,y_{k-3}$ be the $1$-neighbors of $u$ in $G$
corresponding to its ring neighbors. Suppose that $uap$ is a
$1$-thread, where $a$ is its $2$-vertex and $p$ is its other
anchor, and let $x$ be the remaining neighbor of $u$. Since the
$k-3$ $1$-neighbors of $u$ have already been listed, $x$ is not a
$1$-vertex. Thus $d_G(p)\ge3$ and $d_G(x)\ge2$.

If $Q$ is a $2$-thread, write $Q=uvwz$, and put \(S=\{u,v,w,a,y_1,\ldots,y_{k-3}\}\), where $z$ is the other anchor with $d_G(z)\ge 4$ by \textbf{Lemma~\ref{lem:structure-1}}\textbf{(i)}.
If $Q$ is a $3$-thread, write $Q=uvwzr$,
and put
\( S=\{u,v,w,z,a,y_1,\ldots,y_{k-3}\} \), where $r$ is the other anchor with $d_G(r)\ge 4$ by \textbf{Lemma~\ref{lem:structure-1}}\textbf{(i)}.
Take an admissible $L$-coloring $\varphi$ of $G-S$. 

When $Q$ is a $3$-thread, set $\rho=\varphi(r)$ and $\tau=\varphi^*(r)$, choose
$\zeta\in L(z)\setminus\{\rho,\tau\}$, and extend $\varphi$ by
$\varphi(z)=\zeta$. Since $\zeta\neq\tau$,
fix $\varphi^*(z)=\rho$. The resulting coloring is admissible on
$G-(S\setminus\{z\})$. 
After this preliminary step, in either case, we always denote the uncolored set $\{u,v,w,a,y_1,\ldots,y_{k-3}\}$ by $S$, the admissible $L$-coloring of $G-S$ by $\varphi$, where $\varphi^*(p)$, $\varphi^*(z)$ and
$\varphi^*(x)$ are defined.


Put $\beta=\varphi(x)$ and $\eta=\varphi(p)$, and let
$C=L(u)\setminus\{\beta,\varphi^*(x),\eta\}$. Since
$|L(u)|=k+1$, we have
\begin{equation}\label{eq:thread-reduction-C-size}
|C|\ge k-2.
\end{equation}
For $\alpha\in C$, choose
$c_\alpha\in L(w)\setminus\{\alpha,\varphi(z),\varphi^*(z)\}$,
and define
\(A_\alpha(y_i)=L(y_i)\setminus\{\alpha\}\) for \(i\in[k-3]\),
\[
A_\alpha(a)=L(a)\setminus\{\alpha,\varphi(p),\varphi^*(p)\},
\qquad
A_\alpha(v)=L(v)\setminus\{\alpha,c_\alpha,\varphi(z)\}.
\]
All these sets are nonempty. Let
\[
\mathcal A_\alpha=
\bigl(A_\alpha(y_1),\ldots,A_\alpha(y_{k-3}),
A_\alpha(a),A_\alpha(v)\bigr),
\qquad B=(\beta).
\]

We claim that $\mathcal A_\alpha$ has no \textnormal{PCF}-representative system with respect to $B$. Otherwise, using such a system to color the corresponding neighbors of $u$, together with setting
$\varphi(u)=\alpha$ and $\varphi(w)=c_\alpha$, extends $\varphi$ to
a \textnormal{PCF} $L$-coloring of $G$, a contradiction.

For $\alpha\in C$, let $q_\alpha$ be the number of singleton sets
in $\mathcal A_\alpha$, and let $D_\alpha$ be their maximum size.
Put
\[
C_1=\{\alpha\in C:D_\alpha=1\},
\qquad
C_{\ge3}=\{\alpha\in C:D_\alpha\ge3\}.
\]
Applying \textbf{Lemma~\ref{lem:pcf-representative-system}} to $\mathcal A_\alpha \circ B$  yields that  $D_\alpha\le r+s+\left\lfloor\frac{q_\alpha-s}{2}\right\rfloor$, where $s=1$ and $r=0$.
It follows that $q_{\alpha}=k-1\ge 3$ when $D_\alpha=1$, $q_{\alpha}\ge 3$ when $D_\alpha=2$, and $q_{\alpha}\ge 5$ when $D_\alpha\ge 3$. Hence
\begin{equation}\label{eq:thread-reduction-singleton-lower}
\sum_{\alpha\in C}q_\alpha
\ge3|C|+(k-4)|C_1|+2|C_{\ge3}|.
\end{equation}
For each $i\in[k-3]$, the set $A_\alpha(y_i)$ is a singleton for
at most two colors $\alpha\in C$. The set $A_\alpha(a)$ is a
singleton for at most two colors, since singletonhood forces
$\alpha,\varphi(p),\varphi^*(p)$ to be three distinct members of
the $4$-set $L(a)$. Similarly, $A_\alpha(v)$ is a singleton for at
most three colors. Hence
\begin{equation}\label{eq:thread-reduction-singleton-upper}
\sum_{\alpha\in C}q_\alpha\le2(k-3)+2+3=2k-1.
\end{equation}
Combining \eqref{eq:thread-reduction-C-size}--\eqref{eq:thread-reduction-singleton-upper} yields
\begin{equation*}
k-5+(k-4)|C_1|+2|C_{\ge3}|\le0.
\end{equation*}
Thus $k\le5$.

Suppose that $k=5$. Equality holds throughout \eqref{eq:thread-reduction-C-size}--\eqref{eq:thread-reduction-singleton-upper}. Thus
$|C|=3$, $C_1=C_{\ge3}=\emptyset$, each
$A_\alpha(y_i)$ is a singleton for exactly two colors of $C$,
$A_\alpha(a)$ is a singleton for exactly two colors of $C$, and
$A_\alpha(v)$ is a singleton for every $\alpha\in C$. Consequently
$L(y_i)\subseteq C$ for each $i\in[2]$. If
$T_a=\{\alpha\in C:|A_\alpha(a)|=1\}$, then $|T_a|=2$ and
$L(a)=T_a\cup\{\varphi(p),\varphi^*(p)\}$, with four distinct
colors. Hence $A_\alpha(a)\subseteq C$ for every $\alpha\in C$.
Moreover, the singleton equality for $A_\alpha(v)$, for every
$\alpha\in C$, implies
$L(v)=C\cup\{\varphi(z)\}$. Therefore every set in
$\mathcal A_\alpha$ is contained in $C$. Since $\beta\notin C$,
arbitrary representatives are different from $\beta$, which yields a \textnormal{PCF} $L$-coloring of $G$, a contradiction.

Hence $k=4$. Write $y=y_1$ and $C=\{c_1,c_2\}$. The bounds force
$q_\alpha=3$ for each $\alpha\in C$, so all three sets in
$\mathcal A_\alpha$ are singletons. It follows that
\[
L(y)=C,
\qquad
L(a)=C\cup\{\varphi(p),\varphi^*(p)\},
\]
where the four colors in the second set are distinct. Thus, for
$\alpha=c_i$ and $j=3-i$, the representatives on $y$ and $a$ are
both $c_j$.

For $i\in[2]$, write $\omega_i=c_{c_i}$ and
$A_{c_i}(v)=\{d_i\}$. If $d_i\neq\beta$, then
$(c_j,c_j,d_i,\beta)$ has a unique color, contrary to the
nonexistence of a \textnormal{PCF}-representative system. Hence
$A_{c_1}(v)=A_{c_2}(v)=\{\beta\}$. Since $|L(v)|=4$, comparison of
\[
L(v)\setminus\{c_i,\omega_i,\varphi(z)\}=\{\beta\}
\qquad(i\in[2])
\]
gives
\[
\omega_1=c_2,
\qquad
\omega_2=c_1,
\qquad
L(v)=C\cup\{\varphi(z),\beta\}.
\]
In particular, the four colors in the last set are distinct, and
$c_1,c_2\notin\{\varphi(z),\varphi^*(z)\}$.

It remains to exclude the possibility that $Q$ is a $3$-thread.
Recall that in this case $Q=uvwzr$,
$\rho=\varphi(r)$, $\tau=\varphi^*(r)$,
$\varphi(z)=\zeta$, and $\varphi^*(z)=\rho$. There exist
$i\in[2]$, with $j=3-i$, and a color
\[
d\in L(z)\setminus\{c_j,\rho,\tau\}
\quad\text{with}\quad d\neq\zeta.
\]
Indeed, otherwise
$L(z)\subseteq\{\zeta,c_2,\rho,\tau\}$ and
$L(z)\subseteq\{\zeta,c_1,\rho,\tau\}$; the intersection of
these two sets has size at most three, contrary to $|L(z)|=4$.

Set $\varphi(u)=c_i$,
$\varphi(y)=\varphi(a)=\varphi(w)=c_j$, 
$\varphi(v)=\zeta$, and recolor $z$ with $d$. 
It is easy to verify that the resulting coloring is a \textnormal{PCF} $L$-coloring of $G$, a contradiction that completes the proof.
\end{proof}

\begin{lemma}\label{lem:one-thread-good-neighbor}
Let $u$ be a good core $4$-vertex with $n_{\mathrm{ring}}(u)=1$,
and suppose that $u$ is incident with a $2$-thread $Q$. If $u$ is
incident with a $1$-thread, then it is incident with exactly one
$1$-thread and has no bad core neighbor.
\end{lemma}

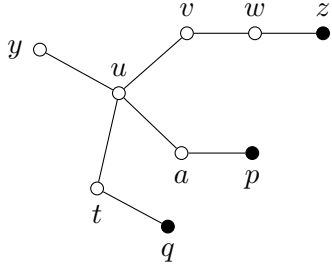
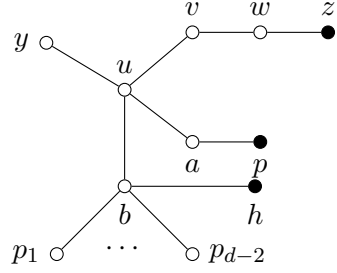
\begin{figure}[htbp]
\centering
\begin{subfigure}[t]{0.48\textwidth}
\centering
\begin{tikzpicture}[scale=0.72]
\tikzset{
deleted vertex/.style={draw,shape=circle,fill=white,inner sep=1.7pt},
retained vertex/.style={draw,shape=circle,fill=black,inner sep=1.7pt}
}
\node[deleted vertex,label=above:$u$] (u) at (0,0) {};
\node[deleted vertex,label=above:$v$] (v) at (1.25,1.1) {};
\node[deleted vertex,label=above:$w$] (w) at (2.5,1.1) {};
\node[retained vertex,label=above:$z$] (z) at (3.75,1.1) {};
\node[deleted vertex,label=left:$y$] (y) at (-1.45,0.8) {};
\node[deleted vertex,label=below:$a$] (a) at (1.15,-1.1) {};
\node[retained vertex,label=below:$p$] (p) at (2.45,-1.1) {};
\node[deleted vertex,label=below:$t$] (t) at (-0.4,-1.75) {};
\node[retained vertex,label=below:$q$] (q) at (0.9,-2.45) {};
\draw (u)--(v)--(w)--(z);
\draw (u)--(y);
\draw (u)--(a)--(p);
\draw (u)--(t)--(q);
\end{tikzpicture}
\caption{Two $1$-threads $uap$ and $utq$.}
\end{subfigure}
\hfill
\begin{subfigure}[t]{0.49\textwidth}
\centering
\begin{tikzpicture}[scale=0.69]
\tikzset{
deleted vertex/.style={draw,shape=circle,fill=white,inner sep=1.7pt},
retained vertex/.style={draw,shape=circle,fill=black,inner sep=1.7pt}
}
\node[deleted vertex,label=above:$u$] (u) at (0,0) {};
\node[deleted vertex,label=above:$v$] (v) at (1.3,1.1) {};
\node[deleted vertex,label=above:$w$] (w) at (2.6,1.1) {};
\node[retained vertex,label=above:$z$] (z) at (3.9,1.1) {};
\node[deleted vertex,label=below:$a$] (a) at (1.3,-1.0) {};
\node[retained vertex,label=below:$p$] (p) at (2.6,-1.0) {};
\node[deleted vertex,label=left:$y$] (y) at (-1.5,0.9) {};
\node[deleted vertex,label=below:$b$] (b) at (0,-1.85) {};
\node[retained vertex,label=below:$h$] (h) at (2.5,-1.85) {};
\node[deleted vertex,label=left:$p_1$] (p1) at (-1.3,-3.15) {};
\node[draw=none] at (0,-3.05) {$\cdots$};
\node[deleted vertex,label=right:$p_{d-2}$] (pd) at (1.3,-3.15) {};
\draw (u)--(v)--(w)--(z);
\draw (u)--(a)--(p);
\draw (u)--(y);
\draw (u)--(b)--(h);
\draw (b)--(p1);
\draw (b)--(pd);
\end{tikzpicture}
\caption{A bad remaining core neighbor $b$.}
\end{subfigure}
\caption{Configurations excluded in
\textbf{Lemma~\ref{lem:one-thread-good-neighbor}}. Hollow vertices are deleted.}
\label{fig:one-thread-good-neighbor}
\end{figure}

\begin{proof}
Let $y$ be the unique $1$-neighbor of $u$ corresponding to its ring neighbor. Write $Q=uvwz$ where $v,w$ are its two $2$-vertices and $z$ is its other anchor. 
By \textbf{Lemma~\ref{lem:structure-1}(i)}, $d_G(z)\ge 4$.

Suppose first that $u$ is incident with at least two $1$-threads.
Since $d_G(u)=4$, the neighbors $y$ and $v$ leave exactly two
remaining incidences, so $u$ is incident with precisely two
$1$-threads. Write them as $uap$ and $utq$, where $a,t$ are the
$2$-vertices and $p,q$ the other anchors (with degree at least $3$). Put
$S=\{u,v,w,a,t,y\}$ and take an admissible $L$-coloring $\varphi$
of $G-S$. The vertices $z,p,q$ are non-isolated in $G-S$; hence $\varphi^*(p),\varphi^*(q),\varphi^*(z)$ are defined.

Let $C_0=L(u)\setminus\{\varphi(p),\varphi(q)\}$. Then
$|C_0|\ge3$. For $\alpha\in C_0$, choose
$c_\alpha\in L(w)\setminus\{\alpha,\varphi(z),\varphi^*(z)\}$
and consider the family $\mathcal A_\alpha$ of available color sets for the neighbors of $u$ where
\[
\mathcal A_\alpha=\{ L(y)\setminus\{\alpha\}, L(a)\setminus\{\alpha,\varphi(p),\varphi^*(p)\},L(t)\setminus\{\alpha,\varphi(q),\varphi^*(q)\}, 
L(v)\setminus\{\alpha,c_\alpha,\varphi(z)\}\}.
\]
Let $B=\emptyset$. There does not exist a \textnormal{PCF}-representative system for $\mathcal A_\alpha \circ B$.  By \textbf{Lemma~\ref{lem:pcf-representative-system}}, all four sets in  $\mathcal A_\alpha $ must be singletons for each $\alpha\in C_0$. On the other
hand, the four sets can be singletons for at most $2,2,2,3$ values
of $\alpha$, respectively. Therefore $4|C_0|\le9$, contradicting
$|C_0|\ge3$. Thus $u$ is incident with exactly one $1$-thread.

Write this unique $1$-thread as $uap$, and let $b$ be the remaining
neighbor of $u$. Suppose that $b$ is bad. Put
$d=d_H(b)=d_G(b)$, let $p_1,\ldots,p_{d-2}$ be the $1$-neighbors
of $b$ corresponding to its ring neighbors, and let $h$ be its
other core neighbor. By \textbf{Lemma~\ref{lem:structure-5}} and
\textbf{Corollary~\ref{cor:structure-4}}\textbf{(iii)}, $b$ is
incident with neither a $1$-thread nor a $2^+$-thread. Hence $d_G(h)=d_H(h)\ge 3$. Delete
\[
S=\{u,v,w,a,b,y\}\cup\{p_j:j\in[d-2]\}
\]
and take an admissible $L$-coloring $\varphi$ of $G-S$. The
vertices $z,p,h$ are non-isolated in $G-S$. 
Hence $\varphi^*(p),\varphi^*(h),\varphi^*(z)$ are defined.

Apply \textbf{Lemma~\ref{lem:bad-two-color-extension}} to $b$, where its
$1$-neighbors $p_1,\ldots,p_{d-2}$, and the colored neighbor $h$ have the same meaning.
So $b$ has two choices $\gamma_1,\gamma_2$ and a witness color $\lambda$
(a unique color in the subsequent coloring).
Let $P=\{\gamma_1,\gamma_2\}$. 
Put
$C_0=L(u)\setminus\{\varphi(p),\lambda\}$, so $|C_0|\ge3$.
For $\alpha\in C_0$, choose
$c_\alpha\in L(w)\setminus\{\alpha,\varphi(z),\varphi^*(z)\}$
and consider the family $\mathcal B_\alpha$ of available color sets for the neighbors of $u$ where
\[
\mathcal B_\alpha=\{L(y)\setminus\{\alpha\},
L(a)\setminus\{\alpha,\varphi(p),\varphi^*(p)\},
P\setminus\{\alpha\},
L(v)\setminus\{\alpha,c_\alpha,\varphi(z)\}\}.
\]
Let $B=\emptyset$. There does not exist a \textnormal{PCF}-representative system for $\mathcal B_\alpha \circ B$. Otherwise, we color the neighborhood of $u$ using such a system first, where $\varphi(b)=\gamma_i$ for some $i\in [2]$, and then color the neighborhood of $b$ using the \textnormal{PCF}-representative system corresponding to $\gamma_i$ as described in \textbf{Lemma~\ref{lem:bad-two-color-extension}}. The resulting coloring is a \textnormal{PCF} $L$-coloring of $G$, a contradiction.  

By \textbf{Lemma~\ref{lem:pcf-representative-system}}, all the sets in $\mathcal B_\alpha$ are singleton sets for every $\alpha\in C_0$, while the
four sets above are singletons for at most $2, 2, 2, 3$ choices of
$\alpha$, respectively. Hence double counting gives $4|C_0|\le9$, a contradiction. Therefore
$b$ is good, and $u$ has no bad core neighbor.
\end{proof}

\begin{lemma}\label{lem:no-one-thread-bad-neighbor}
Let $u$ be a good core $k$-vertex satisfying
\(n_{\mathrm{ring}}(u)=k-3\),
and suppose that $u$ is incident with a $2^+$-thread $Q$. If $u$
is incident with no $1$-thread, then $u$ has no bad core neighbor.
\end{lemma}

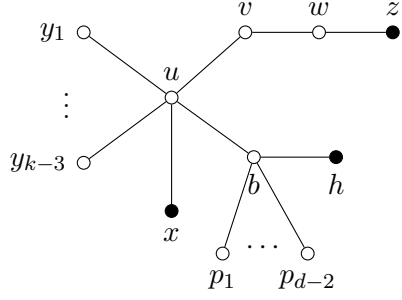
\begin{figure}[htbp]
\centering
\begin{tikzpicture}[scale=0.75]
\tikzset{
deleted vertex/.style={draw,shape=circle,fill=white,inner sep=1.7pt},
retained vertex/.style={draw,shape=circle,fill=black,inner sep=1.7pt}
}
\node[deleted vertex,label=above:$u$] (u) at (0,0) {};
\node[deleted vertex,label=above:$v$] (v) at (1.3,1.15) {};
\node[deleted vertex,label=above:$w$] (w) at (2.6,1.15) {};
\node[retained vertex,label=above:$z$] (z) at (3.9,1.15) {};
\node[retained vertex,label=below:$x$] (x) at (0,-2.0) {};
\node[deleted vertex,label=below:$b$] (b) at (1.45,-1.05) {};
\node[retained vertex,label=below:$h$] (h) at (2.9,-1.05) {};
\node[deleted vertex,label=left:$y_1$] (y1) at (-1.55,1.15) {};
\node[draw=none] at (-1.85,0) {$\vdots$};
\node[deleted vertex,label=left:$y_{k-3}$] (yk) at (-1.55,-1.15) {};
\node[deleted vertex,label=below:$p_1$] (p1) at (0.9,-2.75) {};
\node[draw=none] at (1.65,-2.65) {$\cdots$};
\node[deleted vertex,label=below:$p_{d-2}$] (pd) at (2.4,-2.75) {};
\draw (u)--(v)--(w)--(z);
\draw (u)--(x);
\draw (u)--(b)--(h);
\draw (u)--(y1);
\draw (u)--(yk);
\draw (b)--(p1);
\draw (b)--(pd);
\end{tikzpicture}
\caption{Configuration excluded in
\textbf{Lemma~\ref{lem:no-one-thread-bad-neighbor}}. The path $uvwz$ is the
initial segment of $Q$: $z$ is the other anchor for a $2$-thread and the third
$2$-vertex for a $3$-thread. Hollow vertices are deleted.}
\label{fig:no-one-thread-bad-neighbor}
\end{figure}

\begin{proof}
By \textbf{Lemma~\ref{lem:structure-1}}\textbf{(i)}, $k\ge4$.
By \textbf{Corollary~\ref{cor:structure-4}}\textbf{(ii)}, $Q$ is
the unique $2^+$-thread incident with $u$, and
\textbf{Lemma~\ref{lem:structure-1}}\textbf{(ii)} implies that $Q$
is a $2$-thread or a $3$-thread. Let
$y_1,\ldots,y_{k-3}$ be the $1$-neighbors of $u$ in $G$
corresponding to its ring neighbors, and write $uvwz$ for the
initial segment of $Q$, where $v,w$ are the first two $2$-vertices. 

Since $u$ has no $1$-thread, its two remaining core neighbors are
$3^+$-vertices: a core $2$-neighbor would lie on a thread incident
with $u$ by \textbf{Lemma~\ref{lem:structure-1}}\textbf{(iii)},
contrary to the uniqueness of $Q$. Suppose, to the contrary, that
one of these two neighbors, say $b$, is bad, and let $x$ be the
other $3^+$-vertex. Put $d=d_H(b)=d_G(b)$. Let
$p_1,\ldots,p_{d-2}$ be the $1$-neighbors of $b$ corresponding to
its ring neighbors. By \textbf{Lemma~\ref{lem:structure-5}} and
\textbf{Corollary~\ref{cor:structure-4}}\textbf{(iii)}, $b$ is
incident with neither a $1$-thread nor a $2^+$-thread. Hence its
other core neighbor, denoted by $h$, is a $3^+$-vertex and
\(N_G(b)=\{u,h,p_1,\ldots,p_{d-2}\}\).

Set
\[
S=\{u,v,w,b\}\cup\{y_i:i\in[k-3]\}
\cup\{p_j:j\in[d-2]\}
\]
and take an admissible $L$-coloring $\varphi$ of $G-S$. 
Since $d_G(z)\ge 2$ by \textbf{Lemma~\ref{lem:leaf-neighbor}} and $d_G(z)\ge 4$ by \textbf{Lemma~\ref{lem:structure-1} (i)} when $z$ coincides with some of $x, h$,
vertices $x,h,z$ are non-isolated in $G-S$. 
Consequently $\varphi^*(x)$, $\varphi^*(h)$, and
$\varphi^*(z)$ are defined.

Apply \textbf{Lemma~\ref{lem:bad-two-color-extension}} to $b$ with its
$1$-neighbors $p_1,\ldots,p_{d-2}$, and the colored neighbor $h$.
Let $P=\{\gamma_1,\gamma_2\}$ be the resulting $2$-set, and let
$\lambda$ be the corresponding witness color at $b$. Put
$\beta=\varphi(x)$ and
$C=L(u)\setminus\{\beta,\varphi^*(x),\lambda\}$. Since
$|L(u)|=k+1$,
\begin{equation}\label{eq:no-one-thread-C-size}
|C|\ge k-2.
\end{equation}
For $\alpha\in C$, choose
$c_\alpha\in L(w)\setminus\{\alpha,\varphi(z),\varphi^*(z)\}$
and define
\(A_\alpha(y_i)=L(y_i)\setminus\{\alpha\}\) for \(i\in[k-3]\),
\(A_\alpha(b)=P\setminus\{\alpha\}\), and
\(A_\alpha(v)=L(v)\setminus\{\alpha,c_\alpha,\varphi(z)\}\).
Let
\[
\mathcal A_\alpha=
\bigl(A_\alpha(y_1),\ldots,A_\alpha(y_{k-3}),
A_\alpha(b),A_\alpha(v)\bigr),
\qquad B=(\beta).
\]
A \textnormal{PCF}-representative system for $\mathcal A_\alpha$ with respect to
$B$, together with the corresponding \textnormal{PCF}-representative system described in \textbf{Lemma~\ref{lem:bad-two-color-extension}},
would yield a \textnormal{PCF} $L$-coloring of $G$, a contradiction. 
Hence no such system exists.

Let $q_\alpha$ and $D_\alpha$ denote the number of singleton sets
and the maximum set size in $\mathcal A_\alpha$, respectively, and
put
\[
C_1=\{\alpha\in C:D_\alpha=1\},
\qquad
C_{\ge3}=\{\alpha\in C:D_\alpha\ge3\}.
\]
By \textbf{Lemma~\ref{lem:pcf-representative-system}},
\begin{equation}\label{eq:no-one-thread-singleton-lower}
\sum_{\alpha\in C}q_\alpha
\ge3|C|+(k-4)|C_1|+2|C_{\ge3}|.
\end{equation}
On the other hand, the sets $A_\alpha(y_i)$, $A_\alpha(b)$, and
$A_\alpha(v)$ are singletons for at most $2$, $2$, and $3$
choices of $\alpha$, respectively. Thus, in total 
\begin{equation}\label{eq:no-one-thread-singleton-upper}
\sum_{\alpha\in C}q_\alpha\le2k-1.
\end{equation}
Combining \eqref{eq:no-one-thread-C-size}--\eqref{eq:no-one-thread-singleton-upper} gives
\begin{equation*}
k-5+(k-4)|C_1|+2|C_{\ge3}|\le0.
\end{equation*}
Hence $k\le5$.

Suppose that $k=5$. Equality holds throughout \eqref{eq:no-one-thread-C-size}--\eqref{eq:no-one-thread-singleton-upper}, so
$|C|=3$ and $C_1=C_{\ge3}=\emptyset$. Equality in the individual
upper bounds implies $L(y_i)\subseteq C$ for every $i\in[2]$,
$P\subseteq C$, and $A_\alpha(v)$ is a singleton for every
$\alpha\in C$. The last condition, together with $|L(v)|=4$,
gives $L(v)=C\cup\{\varphi(z)\}$. Thus every set in
$\mathcal A_\alpha$ is contained in $C$. Since $\beta\notin C$,
arbitrary representatives for each set in $\mathcal A_\alpha$ are different from $\beta$, and form a \textnormal{PCF}-representative system for $\mathcal A_\alpha$ with respect to
$B$. Such a system, together with the corresponding \textnormal{PCF}-representative system described in \textbf{Lemma~\ref{lem:bad-two-color-extension}}, would yield a \textnormal{PCF} $L$-coloring of $G$, a contradiction. 

We are left with $k=4$. Let $y=y_1$ and $C=\{c_1,c_2\}$. Equality
in the bounds gives
\begin{equation*}
L(y)=C,
\qquad
P=C,
\end{equation*}
and all three available sets are singletons for both colors of $C$.
The singleton from
$A_{c_i}(v)$ must be $\beta$. Otherwise there exists a \textnormal{PCF}-representative system for $\mathcal A_{c_i} \circ B$ where $\beta$ is a unique color, a contradiction. 
 Hence
\begin{equation*}
A_{c_1}(v)=A_{c_2}(v)=\{\beta\}.
\end{equation*}
Since $|L(v)|=4$, comparison of
$L(v)\setminus\{c_i,c_{c_i},\varphi(z)\}=\{\beta\}$ for
$i=1,2$ yields
\begin{equation}\label{eq:no-one-thread-v-list}
L(v)=C\cup\{\varphi(z),\beta\},
\end{equation}
\begin{equation*}
c_{c_1}=c_2,
\qquad
c_{c_2}=c_1.
\end{equation*}
The four colors in \eqref{eq:no-one-thread-v-list} are distinct, and the definition of
$c_{c_i}$ gives
\begin{equation*}
c_1,c_2\notin\{\varphi(z),\varphi^*(z)\}.
\end{equation*}
Moreover, since $|C|=2$ and $|L(u)|=5$, the three colors
$\beta,\varphi^*(x),\lambda$ are pairwise distinct members of
$L(u)$; in particular,
\begin{equation*}
\lambda\neq\varphi^*(x).
\end{equation*}

Set $\varphi(u)=\lambda$,
$\varphi(y)=\varphi(w)=c_1$, and $\varphi(v)=c_2$. Now the neighborhood of $u$ has three unique colors. Reconsider the color of $b$ and its uncolored neighbors. 

Put
\begin{equation*}
\delta=\varphi(h),\qquad
\theta=\varphi^*(h),\qquad
R=L(b)\setminus\{\lambda,\delta,\theta\}.
\end{equation*}
For $c\in R$ and $j\in[d-2]$, let
\(A_c(p_j)=L(p_j)\setminus\{c\}\), and put
\(B=(\lambda,\delta)\).
Thus it suffices to find some $c\in R$ for which
\(\bigl(A_c(p_1),\ldots,A_c(p_{d-2})\bigr)\)
has a \textnormal{PCF}-representative system with respect to $B$, since
such a system yields a \textnormal{PCF} $L$-coloring of $G$, a contradiction.

Suppose that no such $c$ exists. For $c\in R$, let $q_c$ be the
number of singleton sets among the $A_c(p_j)$, and let
\(D_c=\max_{j\in[d-2]}|A_c(p_j)|\).
Since each $p_j$ has a $2$-list, $D_c\le2$.
By  \textbf{Lemma~\ref{lem:pcf-representative-system}}, the failure of a
\textnormal{PCF}-representative system implies $q_c\ge2$ whether $\lambda=\delta$ or not.
Therefore
\[
2|R|
\le
\sum_{c\in R}q_c.
\]
On the other hand, for each $j$, the set $A_c(p_j)$ is a singleton
only when $c\in L(p_j)$. Hence each $p_j$ contributes at most two
singleton incidences, and consequently
\begin{equation}\label{eq:bad-fixed-color-singleton-upper}
\sum_{c\in R}q_c\le2(d-2).
\end{equation}

If $\lambda=\delta$, then $|R|\ge d-1$, contradicting \eqref{eq:bad-fixed-color-singleton-upper}.
Thus $\lambda\neq\delta$. Now $|R|\ge d-2$, so equality must hold
throughout \eqref{eq:bad-fixed-color-singleton-upper}. In particular, each $p_j$ contributes exactly
two singleton incidences, and hence
\[
L(p_j)\subseteq R=L(b)\setminus\{\lambda,\delta,\theta\}
\qquad (j\in[d-2]).
\]

Fix any $c\in R$. For \(j\in[d-2]\), choose \(a_j\in L(p_j)\setminus\{c\}\). Then none of the colors $a_j$ 
equals $\lambda$ or $\delta$. As $\lambda\neq\delta$, both $\lambda$ and
$\delta$ therefore occur exactly once in
\((a_1,\ldots,a_{d-2},\lambda,\delta)\).
This gives a \textnormal{PCF}-representative system with respect to $B$, a
contradiction.
\end{proof}

Combining \textbf{Lemmas~\ref{lem:thread-reduction},~\ref{lem:one-thread-good-neighbor}, and~\ref{lem:no-one-thread-bad-neighbor}} gives the following corollary.
\begin{lemmacorollary}\label{cor:final-structure}
Let $u$ be a good core $k$-vertex satisfying
\(n_{\mathrm{ring}}(u)=k-3\),
and suppose that $u$ is incident with a $2^+$-thread $Q$.
\begin{itemize}
    \item[\textnormal{\textbf{(i)}}]
    If $Q$ is a $3$-thread, then $u$ is incident with no
    $1$-thread and has no bad core neighbor.

    \item[\textnormal{\textbf{(ii)}}]
    If $Q$ is a $2$-thread, then
    $u$ is incident with at most one $1$-thread and has no bad core neighbor.
\end{itemize}
\end{lemmacorollary}



\section{Proof of Theorem~\ref{th1}}\label{s2}

Let $G$, $L$, $K$, and $H=H_5(G)$ be as fixed in
\textbf{Section~\ref{s:structural}}. In particular,
\(\operatorname{mad}(H)<12/5\), and all structural conclusions proved
there apply.

We now discharge on \(H\). Give each vertex $v$ initial charge
\(\mu(v)=d_H(v)\) and apply the following rules.
\begin{itemize}
    \item[\textbf{(R1)}]
    Every core $4^+$-vertex sends charge $1$ to each of its ring
    neighbors.

    \item[\textbf{(R2)}]
    If a thread has two distinct anchors, then each anchor sends
    charge \(\frac15\) to every \(2\)-vertex on that thread.

    \item[\textbf{(R3)}]
    If a thread has exactly one anchor, then this anchor sends
    charge \(\frac25\) to every \(2\)-vertex on that thread.

    \item[\textbf{(R4)}]
    Every good core vertex sends charge $\frac15$ to each adjacent
    bad core vertex.
\end{itemize}

Let \(\mu^*(v)\) denote the final charge of $v$. 
We divide the analysis into four cases.

\textbf{Case 1.} $v$ is a $2$-vertex. 

If $v$ is a ring vertex, then
\textbf{Lemma~\ref{lem:ring-structure}}\textbf{(iii) and (iv)} implies that it is on a $4$-thread whose anchors are either one ring $3$-vertex or
two distinct ring $3$-vertices. Accordingly, $v$ receives
$\frac25$ under \textbf{(R3)}, or $\frac15$ from each anchor under
\textbf{(R2)}. If $v$ is a core vertex, then by 
\textbf{Lemma~\ref{lem:structure-1}}\textbf{(iii)}, it is on a
thread with two distinct core anchors. So it receives $\frac15$
from each under \textbf{(R2)}. Thus  we always have    
\( \mu^*(v)=2+\frac25=\frac{12}{5}. \)

\textbf{Case 2.} $v$ is a ring $3$-vertex.

By
\textbf{Lemma~\ref{lem:ring-structure}}\textbf{(i)}, it has one
core neighbor, which is a core $4^+$-vertex by
\textbf{Lemma~\ref{lem:ring-structure}}\textbf{(ii)}. Hence $v$ receives charge $1$ under \textbf{(R1)}. By
\textbf{Lemma~\ref{lem:ring-structure}}\textbf{(iv)}, either $v$
is the sole anchor of one $4$-thread and sends
$4\cdot\frac25=\frac85$ under \textbf{(R3)}, or it is an anchor of exactly two $4$-threads with distinct anchors and sends
$8\cdot\frac15=\frac85$ under
\textbf{(R2)}. Therefore \( \mu^*(v)=3+1-\frac85=\frac{12}{5}. \)

\textbf{Case 3.} $v$ is a good core $k$-vertex.

Put
$q=k-n_{\mathrm{ring}}(v)$. After \textbf{(R1)}, the charge remaining
at $v$ is $q$, and $q\ge3$ because $v$ is good. By
\textbf{Lemma~\ref{lem:ring-structure}}\textbf{(v)}, every thread
incident with $v$ contains only core $2$-vertices; by
\textbf{Lemma~\ref{lem:structure-1}}\textbf{(iii)}, its two
anchors are distinct. Hence only \textbf{(R2)}, not
\textbf{(R3)}, applies to threads incident with $v$. Moreover,
\textbf{Lemma~\ref{lem:structure-1}}\textbf{(ii)} shows that each
such thread contains at most three $2$-vertices.

Let $t=n_{2^+\text{-thread}}(v)$, let $a$ be the number of
$1$-threads incident with $v$, and let $b$ be the number of bad core
neighbors of $v$. Then 
\begin{equation}\label{eq:good-core-incidence-count}
a+b+t\le q.
\end{equation}
Moreover, \textbf{Lemma~\ref{lem:structure-4}} gives
\begin{equation}\label{eq:good-core-thread-count}
t\le q-2.
\end{equation}
The total charge sent by $v$ under
\textbf{(R2)} and \textbf{(R4)} is at most
\begin{equation}\label{eq:good-core-outgoing-charge}
\frac{3t+a+b}{5}.
\end{equation}
 From \eqref{eq:good-core-incidence-count}--\eqref{eq:good-core-outgoing-charge}, \(3t+a+b\le3t+(q-t)=q+2t\le3q-4. \)
Consequently,
\[
\mu^*(v)\ge q-\frac{3q-4}{5}
=\frac{2q+4}{5}.
\]
When $q\ge4$, \(\mu^*(v)\ge\frac{12}{5}\).

It remains to consider $q=3$. By \eqref{eq:good-core-thread-count}, $t\le1$. If $t=0$,
then \eqref{eq:good-core-incidence-count} gives $a+b\le3$, and
$\mu^*(v)\ge3-\frac35=\frac{12}{5}$. Suppose that $t=1$.
Then $n_{\mathrm{ring}}(v)=k-3$. Let $m\in\{2,3\}$ be the number
of $2$-vertices on the $2^+$-thread. The charge sent under
\textbf{(R2)} and \textbf{(R4)} is at most
\(\frac{m+a+b}{5}\).
If $m=3$, \textbf{Corollary~\ref{cor:final-structure}}\textbf{(i)} gives
$a=b=0$. If $m=2$,
\textbf{Corollary~\ref{cor:final-structure}}\textbf{(ii)} gives
$a+b\le1$. Thus $m+a+b\le3$ in either case, and
\(
\mu^*(v)\ge3-\frac35=\frac{12}{5}.
\)

\textbf{Case 4.} $v$ is a bad core $k$-vertex. 

Since $n_{\mathrm{ring}}(v)=k-2$, the charge remaining after
\textbf{(R1)} is $2$. By
\textbf{Corollary~\ref{cor:structure-4}}\textbf{(iii)}, $v$ is
incident with no $2^+$-thread, and
\textbf{Lemma~\ref{lem:structure-5}} excludes a $1$-thread. The two remaining neighbors of $v$ are core $3^+$-vertices. 
They are good
by \textbf{Lemma~\ref{lem:structure-6}}, so $v$ receives
$\frac15$ from each under \textbf{(R4)}. Therefore
\(
\mu^*(v)=2+\frac15+\frac15=\frac{12}{5}.
\)

In summary, every vertex of \(H\) has final charge at least \(12/5\). Since the
total charge is preserved,
\[
2|E(H)|=\sum_{v\in V(H)}\mu^*(v)
\ge\frac{12}{5}|V(H)|.
\]
Thus the average degree of $H$ is at least $12/5$, contrary to
$\operatorname{mad}(H)<12/5$ established at the beginning of the
section. This completes the proof of
\textbf{Theorem~\ref{th1}}.
\qed

\section{Proof of Theorem~\ref{th2}}\label{s3}


A graph is a \emph{linear forest} if each of its components is a
path. A coloring is called \emph{superlinear} if it is proper, the
subgraph induced by any two color classes is a linear forest, and the
two neighbors of every $2$-vertex receive distinct colors. Obviously, a superlinear coloring of a subcubic graph is also a \textnormal{PCF} coloring. 

We shall use the following theorem of Liu and Yu. 

\begin{lemma}[\cite{liu2013linear}]\label{lem:liu-yu}
Let $G$ be a connected subcubic graph distinct from $C_5$ and
$K_{3,3}$. If $L$ is a list assignment satisfying
$|L(v)|\ge4$ for every $v\in V(G)$, then $G$ has a superlinear
$L$-coloring.
\end{lemma}

\begin{lemma}\label{k33}
The graph $K_{3,3}$ is \textnormal{PCF} $4$-choosable.
\end{lemma}

\begin{proof}
Let the bipartition of $K_{3,3}$ be
$A=\{a_1,a_2,a_3\}$ and $B=\{b_1,b_2,b_3\}$. Let $L$ be a list
assignment with $|L(v)|\ge4$ for every vertex $v$.

Choose a color $\alpha\in L(a_1)$, and then choose a color
$\beta\in L(a_2)\setminus\{\alpha\}$. Set $\varphi(a_1)=\alpha$ and $\varphi(a_2)=\beta$. For each $b_i\in B$, set
$L'(b_i)=L(b_i)\setminus\{\alpha,\beta\}$. Then
$|L'(b_i)|\ge2$ for each $i\in[3]$. Choose
$\varphi(b_1)=\gamma\in L'(b_1)$, and choose
$\varphi(b_2)=\delta\in L'(b_2)\setminus\{\gamma\}$. Choose any
color $\varphi(b_3)\in L'(b_3)$. Since at most three colors appear
on $B$, choose \(\varphi(a_3)\in L(a_3)\setminus
\{\varphi(b_1),\varphi(b_2),\varphi(b_3)\}\).

No color used on $A$ appears on $B$, so the coloring is proper.
Moreover, $A$ receives at least two colors, since
$\varphi(a_1)\neq\varphi(a_2)$, and $B$ receives at least two
colors, since $\varphi(b_1)\neq\varphi(b_2)$. Every vertex of $A$
has neighborhood $B$, and every vertex of $B$ has neighborhood
$A$. Thus every vertex has a color appearing exactly once in its
neighborhood, and the coloring is a \textnormal{PCF} $L$-coloring of $K_{3,3}$ for any given 4-list assignment $L$.
\end{proof}

\begin{proof}[Proof of Theorem~\ref{th2}]
Suppose, to the contrary, that the theorem is false. Let $G$ be a
counterexample with $|V(G)|$ minimum, and let $L_0$ be a
$\kappa_G$-list assignment for which $G$ has no \textnormal{PCF}
$L_0$-coloring. Obviously, $G$ is connected. For every $v\in V(G)$, choose
$L(v)\subseteq L_0(v)$ with
\(|L(v)|=\kappa_G(v)\).
A \textnormal{PCF} $L$-coloring of $G$ would also be a \textnormal{PCF} $L_0$-coloring, so $G$ has
no \textnormal{PCF} $L$-coloring. The graph $K_1$ is trivially \textnormal{PCF}
$L$-colorable, so $|V(G)|\ge2$.

Let $\emptyset\neq S\subsetneq V(G)$ and let $F$ be a component of
$G-S$ with $F\not\cong C_5$. Since $F$ is connected and subcubic,
and
\[
|L(v)|=\kappa_G(v)\ge\kappa_F(v)
\qquad\text{for every }v\in V(F),
\]
the minimality of $G$ gives a \textnormal{PCF} $L|_F$-coloring of $F$. Hence the
hypotheses of \textbf{Lemma~\ref{lem:leaf-neighbor}} are satisfied.
Every $1$-vertex of $G$ would therefore have a $4^+$-neighbor,
which is impossible because $G$ is subcubic. Since $G$ is connected
and has at least two vertices, it follows that $\delta(G)\ge2$.

Since $G$ is subcubic, every vertex has degree $2$ or $3$. Hence
$|L(v)|\ge4$ for every $v\in V(G)$. If $G=K_{3,3}$, then the
result follows from \textbf{Lemma~\ref{k33}}, a contradiction.
Therefore $G$ is a connected subcubic graph distinct from $C_5$
and $K_{3,3}$.

By \textbf{Lemma~\ref{lem:liu-yu}}, the graph $G$ has a superlinear
$L$-coloring. This coloring is a 
\textnormal{PCF} $L$-coloring of $G$, contradicting the choice of $G$. This
completes the proof.
\end{proof}

\noindent\textbf{Acknowledgement}

This research is supported in part by the National Natural Science
Foundation of China (Nos. 12471330, 12571373).

\noindent\textbf{Declaration on the Use of Generative AI}

All mathematical ideas, problem formulations, proof strategies, arguments, and results presented in this paper were developed by the authors. Generative AI tools were used solely for auxiliary purposes, including checking typographical and linguistic errors, improving the organization and clarity of the manuscript, and drawing attention to possible inconsistencies or gaps that were subsequently examined by the authors. The authors independently verified all statements and proofs and take full responsibility for the content of the paper.

\bibliography{cas-refs}
\end{document}